\documentclass{amsart}
\usepackage[hidelinks]{hyperref}
\usepackage{amsmath,amssymb,amsfonts,amsthm, mathtools}
\usepackage{tikz}
\usepackage{doi,url}
\usepackage{comment}
\usepackage[numbers]{natbib}
\DeclareRobustCommand{\noopsort}[1]{}

\usepackage[inline]{enumitem}
\usepackage{tikz-cd}

\title{Residually Constructible Extensions}

\author{Pietro Freni}
\author{Angus Matthews}

\thanks{
The first author's work was supported by EPSRC Fellowship EP/T018461/1 at the University of Leeds, by the Czech Science Foundation grant 26-22545I, and by the Czech Academy of Sciences CAS (RVO 67985840).
The second author’s work was supported by EPSRC DTP 2224 [EP/W524372/1] (at the University of Leeds)}
\address{Institute of Mathematics Czech Academy of Sciences, Žitná 25
110 00 Praha 1, Czech Republic}
\address{School of Mathematics, University of Leeds, Leeds LS2 9JT, United Kingdom}
 


\makeatletter
\newtheorem*{rep@theorem}{\rep@title}
\newcommand{\newreptheorem}[2]{%
\newenvironment{rep#1}[1]{%
 \def\rep@title{#2 \ref{##1}}%
 \begin{rep@theorem}}%
 {\end{rep@theorem}}}
\makeatother

\newtheorem{theorem}{Theorem}
\newtheorem{lemma}[theorem]{Lemma}
\newtheorem{proposition}[theorem]{Proposition}
\newtheorem{corollary}[theorem]{Corollary}
\newtheorem{fact}[theorem]{Fact}

\newtheorem{claim}{Claim}[theorem]

\theoremstyle{remark}
\newtheorem{remark}[theorem]{Remark}

\theoremstyle{definition}
\newtheorem{definition}[theorem]{Definition}
\newtheorem{example}[theorem]{Example}

\theoremstyle{plain}

\newcommand{\bE}{\mathbb E}
\newcommand{\bF}{\mathbb F}
\newcommand{\bK}{\mathbb K}
\newcommand{\bN}{\mathbb N}
\newcommand{\bP}{\mathbb P}
\newcommand{\bQ}{\mathbb Q}
\newcommand{\bR}{\mathbb R}
\newcommand{\bU}{\mathbb U}

\newcommand{\cF}{\mathcal F}
\newcommand{\cO}{\mathcal O}
\newcommand{\co}{\mathfrak o}

\newcommand{\oa}{\overline{a}}
\newcommand{\ob}{\overline{b}}

\newcommand{\convex}{\mathrm{convex}}
\newcommand{\tame}{\mathrm{tame}}

\newcommand{\defin}{\mathrm{def}}

\newcommand{\RCF}{\mathrm{RCF}}
\newcommand{\RCVF}{\mathrm{RCVF}}

\newcommand{\ox}{\overline{x}}
\newcommand{\oy}{\overline{y}}
\newcommand{\oz}{\overline{z}}

\DeclareMathOperator{\val}{\mathbf{v}}
\DeclareMathOperator{\rv}{\mathbf{rv}}
\DeclareMathOperator{\res}{\mathbf{res}}

\DeclareMathOperator{\std}{\mathrm{st}}
\DeclareMathOperator{\tp}{\mathrm{tp}}
\DeclareMathOperator{\dcl}{\mathrm{dcl}}
\DeclareMathOperator{\CH}{\mathrm{CH}}

\DeclareMathOperator{\Aut}{\mathrm{Aut}}

\newcommand{\Lhp}{(\!(}
\newcommand{\Rhp}{)\!)}

\newcommand{\var}[1]{\mathrm{#1}}

\theoremstyle{plain}
\newcounter{introthmcounter}
\newtheorem{thmA}{Theorem}[introthmcounter]

\newcommand\blfootnote[1]{%
  \begingroup
  \renewcommand\thefootnote{}\footnote{#1}%
  \addtocounter{footnote}{-1}%
  \endgroup
}

\begin{document}

\begin{abstract}
    Let $T$ be a complete o-minimal theory expanding $\RCF$ and $T_\mathrm{convex}$ be the common theory of its models expanded by predicate for a non-trivial $T$-convex valuation ring.
    We call an elementary extension $(\bE, \cO) \prec (\bE_*, \cO_*) \models T_\mathrm{convex}$ \emph{res-constructible} if there is a tuple $\overline{s}$ in $\cO_*$ such that $\bE_*=\dcl(\bE,\overline{s})$, and the projection $\res(\overline{s})$ of $\overline{s}$ in the residue field sort is $\mathrm{dcl}$-independent over the residue field $\res(\bE, \cO)$ of $(\bE, \cO)$.
    We study factorization properties of res-constructible extensions. Our main result is that a res-constructible extension $(\bE, \cO) \prec (\bE_*, \cO_*)$ has the property that all $(\bE_1, \cO_1)$ with $(\bE, \cO) \prec (\bE_1, \cO_1) \prec (\bE_*, \cO_*)$ are res-constructible over $(\bE, \cO)$, if and only if $\bE_*$ has countable $\dcl$-dimension over $\bE$ or its value group $\val(\bE_*, \cO_*)$ is \emph{short} (i.e.\ contains no uncountable well-ordered subset).
    This analysis entails complete answers to \cite[Problem~5.12]{tressl2006pseudo}.
\end{abstract}


\maketitle
\blfootnote{For the purpose of open access, the authors have applied a Creative Commons Attribution (CC BY) licence to any Author Accepted Manuscript version arising from this submission. There is no data associated with this publication.}

\section{Introduction}

Let $T$ be a complete o-minimal $L$-theory expanding the theory $\RCF$ of real closed fields in some language $L$ containing the language of ordered rings. There has been significant work on this class of theories; for example $T$ could be the theory $T_{\exp}$ of the field of reals expanded by the natural exponential function (\cite{wilkie1996model}) or the theory $T_{an, \exp}$ of the reals expanded by the natural exponential and all restricted analytic functions (\cite{dries1994real}).

Recall from \cite[(2.7)]{dries1995t}, that a \emph{$T$-convex} valuation subring $\cO$ of a model $\bE\models T$ is a convex subring closed under all continuous $\emptyset$-definable functions $f: \bE \to \bE$.

By \cite[(3.13) and (3.14)]{dries1995t}, the common theory $T_\convex$ in the language $L_\convex\coloneqq  L \cup \{\cO\}$ of all models of $T$ expanded by a unary predicate $\cO$ interpreted as a non-trivial $T$-convex valuation ring, is complete and weakly o-minimal.

The results in \cite[(2.12)]{dries1995t}, \cite[Thm.~B]{dries1997t}, and \cite[Ch.~12 and 13]{tyne2003t} entail that if $T$ is power bounded with field of exponents $\Lambda$, then for every elementary extension $(\bE, \cO) \preceq (\bE_*, \cO_*)\models T_\convex$ one can find a $\dcl_T$-basis of $\bE_*$ over $\bE$ of the form $(\overline{r}, \overline{v}, \overline{e})$ where:
\begin{enumerate}
    \item $\overline{r}$ is a tuple in $\cO_*$ such that the projection $\res(\overline{r})$ of $\overline{r}$ in the residue field $\res(\bE_*)$ is $\dcl_T$-independent over the residue field $\res(\bE)$ of $(\bE, \cO)$;
    \item $\overline{v}$ is a tuple in $\bE_*$ such that $\val(\overline{v})$ is $\Lambda$-linearly independent over the value group $\val(\bE)$ of $\bE$;
    \item $(\bE_*, \cO_*)$ is an immediate extension of $(\bE\langle \overline{r}, \overline{v}\rangle, \cO_*\cap \bE \langle \overline{r}, \overline{v}\rangle)$, where $\bE\langle \overline{r}, \overline{v}\rangle\coloneqq \dcl(\bE, \overline{r}, \overline{v})$.
\end{enumerate}

Even when $T$ is not power bounded, we will call tuples like $\overline{r}$ in item (1) above \emph{$\cO_*$-res-constructions} over $\bE$ (Definition~\ref{def:res_constructible}). In this paper we will be concerned with \emph{res-constructible} elementary extensions of models of $T_\convex$, that is, elementary extensions $(\bE_*, \cO_*) \succeq (\bE, \cO)$ such that for some $\cO_*$-res-construction $\overline{r}$ over $\bE$, $\bE_* = \bE \langle \overline{r}\rangle$.

Specific instances of res-constructible extensions were considered in \cite{tressl2006pseudo}: more precisely, a res-constructible extension $(\bE, \cO) \preceq (\bE_*, \cO_*)$ such that the underlying residue field extension $\res(\bE, \cO) \preceq \res(\bE_*, \cO_*)$ is a Cauchy-completion, is exactly a \emph{pseudo-completion} of $\bE$ with respect to $\{\cO\}$ in the sense of \cite[Def.~4.2]{tressl2006pseudo}.

After showing that pseudo-completions are unique up to isomorphisms but admit non-surjective self-embeddings, Tressl in \cite[Problem~5.12]{tressl2006pseudo} asks whether given $(\bE, \cO) \preceq (\bE', \cO') \preceq (\bE_1, \cO_1) \preceq (\bE_*, \cO_*)$ such that $\bE_*$ and $\bE'$ are pseudo-completions of $\bE$ with respect to $\{\cO\}$, one has that necessarily $\bE_1$ is also a pseudo-completion of $\bE$ with respect to $\{\cO\}$.

In this paper we will study the more general question of whether res-constructible extensions are closed under taking right-factors, i.e.\ whether (or rather more appropriately ``when'') given a composition of elementary extensions $(\bE, \cO) \preceq (\bE_1, \cO_1) \preceq (\bE_*, \cO_*)$ of models of $T$-convex such that the composite extension $(\bE, \cO) \preceq (\bE_*, \cO_*)$ is res-constructible, one can deduce that $(\bE, \cO) \preceq (\bE_1, \cO_1)$ is also res-constructible.

We show that in general res-constructible extensions are not closed under right factors and give a complete characterization of those res-constructible extensions $(\bE, \cO)\preceq (\bE_*, \cO_*)$ all of whose right-factors $(\bE, \cO) \preceq (\bE_1, \cO_1)$ are res-constructible.

\begin{thmA}\label{introthm:Main}
    Suppose $(\bE, \cO)\preceq (\bE_*, \cO_*)$ is res-constructible. Then the following are equivalent:
    \begin{enumerate}
        \item $\dim_{\dcl}(\bE_*/\bE)$ is countable or the value group $\val(\bE_*, \cO_*)$ of $(\bE_*, \cO_*)$ is short (i.e.\ contains no uncountable well order, cf \cite[p.~88]{rosenstein1982linear}, \cite{aschenbrenner2025analytic}, \cite{aschenbrenner2025short});
        \item for all $(\bE, \cO) \prec (\bE_1, \cO_1) \prec (\bE_*, \cO_*)$, the extension $(\bE, \cO) \prec (\bE_1, \cO_1)$ is res-constructible;
        \item for all $(\bE, \cO) \prec (\bE_1, \cO_1) \prec (\bE_*, \cO_*)$, with $\res(\bE_1, \cO_1)=\res(\bE_*, \cO_*)$, the extension $(\bE, \cO) \prec (\bE_1, \cO_1)$ is res-constructible.
    \end{enumerate}

\end{thmA}

This implies that the answer to \cite[Problem~5.12]{tressl2006pseudo} is affirmative if and only if the value group $\val(\bE_*, \cO_*)$ does not contain uncountable well orders or $\res(\bE, \cO)$ has a Cauchy-completion of countable $\dcl$-dimension (Section~\ref{sec:Tressl_problem}).

We will now spend some words about the proof of the main implication $(1) \Rightarrow (2)$ of Theorem~\ref{introthm:Main}.

It is not to hard to show that if $(\bE, \cO) \preceq (\bE_*, \cO_*)$ is res-constructible, then for any intermediate $\bE \preceq \bE_1 \preceq \bE_*$, such that $\dim_{\dcl}(\bE_1/\bE)< \aleph_0$, one has that $(\bE, \cO) \preceq (\bE_1, \cO_* \cap \bE_1)$ and $(\bE_1, \cO_* \cap \bE_1) \preceq (\bE_*, \cO_*)$ are both res-constructible extensions (Corollary~\ref{cor:finite_dim_fact}). 
This easily entails that also if $\dim_{\dcl}(\bE_1/\bE)= \aleph_0$, one has that $(\bE, \cO) \preceq (\bE_1, \cO_* \cap \bE_1)$ is res-constructible however it is also not hard to show that $(\bE_1, \cO_* \cap \bE_1) \prec (\bE_*, \cO_*)$ doesn't need to be res-constructible anymore.

We will prove $(1) \Rightarrow (2)$ of Theorem~\ref{introthm:Main} by building a $\cO_*$-res-construction $\overline{s}$ for $\bE_1$ over $\bE$ as an ordinal indexed sequence $\overline{s}\coloneqq (s_i :i <\lambda)$. By the already mentioned Corollary~\ref{cor:finite_dim_fact} about factors with finite dimension, the main problem in the transfinite construction of $\overline{s}$ is at limit steps: in fact it may happen that even if $\bE_1$ was res-constructible over $\bE$, for some $\cO_*$-res-construction, say $(s_i: i < \omega)$, it is not anymore res-constructible over $\bE \langle s_i: i< \omega\rangle$. This difficulty will be overcome by setting up the induction so that at each limit step the res-construction $(s_i: i< \omega \cdot \alpha)$ satisfies both a model theoretic orthogonality condition and a closure condition (Definition~\ref{def:strong}) relative to a given $\cO_*$-res-construction $\overline{r}$ of $\bE_*$ over $\bE$.

Subsequent to proving Theorem A, in Section ~\ref{sec:Tressl_problem}, we will turn to \cite[Problem~5.12]{tressl2006pseudo}, and give a complete answer, dependent on the properties of $\bE_*$.
We will also show that in the power-bounded setting, this dependence can be reduced to considering properties of $\bE$. Then, in Section ~\ref{sec:removing_ambient_res_construction}, we will discuss alternative characterisations of res-constructibility, and construct an example showing that res-constructibility is not a `local' property.

\subsection*{Acknowledgments}
The authors would like to thank Marcus Tressl for initially posing this problem, in a problem session in the 32${}^{\text{nd}}$ LYMoTS. We also thank the organisers. Finally, we thank Vincenzo Mantova for many helpful discussions.

\section{Preliminaries and notation}

In the following $T\supseteq \RCF$ will be some fixed complete o-minimal theory in some language $L$ containing the language of ordered rings. If $\bE \models T$ and $x$ is a tuple or a set in some $\bU \succ \bE$, we will denote the definable closure $\dcl_T(\bE \cup x)$ by $\bE \langle x \rangle$. Recall that this is still a model of $T$ because $T$ has definable Skolem functions.
Also, for a subset $S\subseteq \bE$, we will denote the convex hull of $S$ in $\bU$ by $\CH_\bU(S) \coloneqq \{x \in \bU: \exists s_1, s_2 \in S,\; s_1\le x\le s_2\}$. Finally, if $R\subseteq X \times X$ is a binary relation on a set $X$, given subsets $A,B\subseteq X$, we will write $A^{RB}$ for the set $A^{RB}\coloneqq\{x \in A: \forall b \in B,\, xRb\}$; if $B=\{b\}$ is a singleton we will write $A^{Rb}$ for $A^{R\{b\}}$. Usually $R$ will be an order relation or equality.

\subsection{Valuation-related notation.}
Given a field $\bE$ and a valuation ring $\cO$ on $\bE$, we will denote by $\co$ the unique maximal ideal of $\cO$.
The \emph{value group} will be denoted by $(\val(\bE, \cO), +) \coloneqq (\bE^{\neq 0}, \cdot)/(\cO \setminus \co)$ or by $\val_\cO \bE$.
The \emph{valuation} will be the quotient $\val_\cO : \bE^{\neq 0} \to \val(\bE, \cO)$, and will be extended to the whole $\bE$ by setting $\val(0)=\infty$. The value group will be ordered as usual by setting $\val(\bE, \cO)^{\ge 0} \coloneqq  \val_\cO(\cO^{\neq 0})$.

The associated \emph{dominance}, \emph{weak asymptotic equivalence}, \emph{strict dominance}, and \emph{strict asymptotic equivalence} on $\bE$ will be denoted respectively by:
\[\begin{aligned}
x \preceq y &\Longleftrightarrow \val_\cO(x)\ge \val_\cO(y) \Leftrightarrow x/y \in \cO;\\
x \asymp y &\Longleftrightarrow \val_\cO(x)=\val_\cO(y) \Leftrightarrow x/y \in \cO \setminus \co;\\
x \prec y &\Longleftrightarrow \val_\cO(x)>\val_\cO(y) \Leftrightarrow x/y \in \co;\\
x \sim y &\Longleftrightarrow x-y \prec x \Leftrightarrow 1-x/y \in \co.
\end{aligned}\]

The imaginary sort given by the quotient under the strict asymptotic equivalence $\sim$ will be called \emph{rv-sort} and denoted by $\rv(\bE, \cO)\coloneqq(\bE^{\neq0}, \cdot)/(1+\co)$ or $\rv_\cO(\bE)$, the quotient map will be denoted by $\rv_\cO: \bE^{\neq 0} \to \rv(\bE, \cO)$.

Finally the \emph{residue field} sort will be denoted by $\res_{\cO}(\bE)=\res(\bE,\cO)\coloneqq\cO/\co$ and the corresponding quotient map by $\res_\cO: \cO \to \res(\bE,\cO)$. We extend $\res_\cO$ to a total function $\res_\cO: \bE \to \res(\bE,\cO)$ by setting $\res_\cO$ to be $0$ outside of $\cO$.

The subscript $\cO$ will be omitted when this does not create ambiguity.

\subsection{Main facts on \texorpdfstring{$T$}{T}-convex valuations}
Recall that if $\bE \models T$, a subset $\cO \subseteq \bE$ is said to be a \emph{$T$-convex subring} when it is order-convex and closed under $\emptyset$-definable continuous total unary functions, \cite[2.7]{dries1995t}. Note that such a subset is in particular a subring because $x\mapsto 1$, $x\mapsto -x$, $x\mapsto 2x$, and $x\mapsto x^2$ are all $\emptyset$-definable unary functions and an order convex subset closed under such functions must be a subring.

\begin{fact}[{van den Dries and  Lewenberg \cite[3.6 and 3.7]{dries1995t}}]\label{fact:unary_type_extension_alternative}
    Let $\bU \succeq \bE \models T$ and let $\cO'\subseteq \bU$ and $\cO \subseteq \bE$ be $T$-convex subrings with $\cO'\cap \bE= \cO$. If $x \in \bU \setminus \bE$, then $\cO' \cap \bE \langle x \rangle \in \{\cO_x, \cO_x^-\}$ where
    \[\cO_x\coloneqq\{y \in \bE \langle x \rangle: y<\bE^{>\cO}\} \quad \text{and} \quad \cO_x^-\coloneqq\CH_{\bE \langle x \rangle}(\cO).\]
\end{fact}

\begin{definition}
    Let $L_\convex\coloneqq L \cup \{\cO\}$ be the language $L$ expanded by a unary predicate $\cO$ and let $T_\convex^-$ be the common theory in $L_\convex$ of all expansions $(\bE, \cO)$ of models of $T$ where $\cO$ is interpreted as a $T$-convex subring. Let $T_\convex$ be the theory $T_\convex^- \cup \{\exists x \notin \cO\}$.
\end{definition}

\begin{fact}[{van den Dries and  Lewenberg \cite[3.10 and 3.14]{dries1995t}}]\label{fact:T_convex_is_complete}
    The theory $T_\convex$ is complete and weakly o-minimal. Moreover if $T$ is universally axiomatized and eliminates quantifiers, then $T_\convex$ eliminates quantifiers. 
\end{fact}

\begin{remark}
    Notice that if $T^{\defin}$ is the expansion of $T$ to the language $L^{\defin}$ with a symbol $S_f(\ox)$ for every $T$-definable function $f(\ox)$, by adding axioms that $S_f$ equals $f$ for all $f$, then $T$ is model complete (in fact it is universally axiomatized and eliminates quantifiers). In particular Fact~\ref{fact:T_convex_is_complete} implies that if $(\bE, \cO)\models T_\convex$ and $\bE \succeq \bE'\models T$ is such that $\bE' \not \subseteq \cO$, then $(\bE, \cO) \succeq (\bE', \cO \cap \bE')$.
\end{remark}

\begin{definition}[{\cite[p.187]{marker1994definable}}, {\cite[p.~76 and (1.12)]{dries1997t}}]
    An elementary substructure $\bK \preceq \bE\models T$ is said to be \emph{tame} or \emph{definably Dedekind-complete} (denoted $\bK \preceq_\tame \bE$) if for all $\bE$-definable $S\subseteq \bE$, if $S\cap \bK$ is bounded in $\bK$, then it has a supremum in $\bK$. If $\bK\preceq_\tame \bE$, then there is a unique function $\std_\bK: \CH_\bE(\bK) \to \bK$ with the property that $\std_\bK(x)-x < \bK^{>0}$ for all $x \in \CH_\bE(\bK)$.
\end{definition}

\begin{fact}[{Marker and Steinhorn \cite[Thm.~2.1]{marker1994definable}, see also \cite{andujarguerrero2025marker}
}]
    $\bK \preceq_\tame \bE$ if and only if whenever $c$ is a tuple of $\bE$ then $\tp(c/\bK)$ is definable.
\end{fact}

It is not hard to see that given a $T$-convex subring $\cO \subseteq \bE\models T$, there is $\bK \preceq \bE$ maximal in $\{\bK' \preceq \bE: \bK' \subseteq \cO\}$. 

\begin{fact}[{van den Dries and Lewenberg {\cite[2.12]{dries1995t}}, van den Dries {\cite[Thm.~A]{dries1997t}}\nocite{dries1998corrections}}]\label{fact:res_iso_max}
	If $(\bE, \cO) \models T_\convex^-$, then the following are equivalent for $\bK \preceq \bE$:
    \begin{enumerate}
        \item $\bK$ is maximal in $\{\bK' \preceq \bE: \bK' \subseteq \cO\}$;
        \item $\bK\preceq_\tame \bE$ and $\CH(\bK)=\cO$;
        \item $\co \cap \bK = \{0\}$ and $\std_\bK: \cO \to \bK$ induces an isomorphism between the induced structure on the imaginary sort $\res(\bE, \cO)$ and $\bK$.
    \end{enumerate}
    \begin{proof}
        See \cite[2.12]{dries1995t} for the equivalence of (1) and (2). The equivalence of (2) and (3) is instead in \cite[1.13]{dries1997t}.
    \end{proof}
\end{fact}

\begin{definition}
    We call $\bK \prec \bE$ an \emph{elementary residue-section} for $(\bE, \cO)$ when the equivalent conditions of Fact~\ref{fact:res_iso_max} hold.
\end{definition}

\begin{fact}[{van den Dries and Lewenberg in \cite[5.3]{dries1995t}}]\label{fact:res_dim_inequality}
    Suppose $(\bE,\cO)\prec (\bE \langle x \rangle, \cO')\models T_\convex$, then $\dim_{\dcl}(\res(\bE\langle x \rangle, \cO')/\res(\bE, \cO))\le 1$.
\end{fact}

\subsection{Cuts and \texorpdfstring{$T$}{T}-convex valuations}

Let $(\bU, \cO)\models T_\convex$ and $\bE \prec \bU \models T$.

\begin{definition}
    We call an element $x\in \bU \setminus \bE$:
    \begin{enumerate}
        \item \emph{weakly $\cO$-immediate} (or $\cO$-\emph{wim}) over $\bE$ if $\val_{\cO}(x-\bE)$ has no maximum;
        \item $\cO$-\emph{residual} over $\bE$ if $x \in \cO$ and $\res_{\cO}(x)\notin \res_\cO(\bE)$;
        \item $\cO$-\emph{purely valuational} over $\bE$ if $\bE \langle x \rangle \setminus \bE = \bE + M$ for some $M\subseteq \bE \langle x \rangle$ such that $\val_{\cO}(M) \cap \val_{\cO}(\bE)=\emptyset$. 
    \end{enumerate}
    
    Furthermore we will say that $x$ is \emph{weakly $\cO$-immediately generated} (or \emph{$\cO$-wim generated}) over $\bE$ if $\bE\langle x \rangle = \bE \langle y \rangle$ for some weakly $\cO$-immediate $x$. Similarly we will say that $x$ is \emph{$\cO$-residually generated} over $\bE$ if $\bE\langle x \rangle = \bE \langle y \rangle$ for some residual $y$.
\end{definition}

\begin{remark}
	An element $x \in \bU\setminus \bE$ is \emph{weakly $\cO$-immediate} over $\bE$, if and only if for every $c \in \bE$, $\rv_{\cO}(x-c) \in \rv_{\cO} (\bE)$.
\end{remark}

\begin{fact}[see Thm.~3.10 in \cite{freni2024t}]\label{fact:cuts_tricho}
    For every $x \in \bU \setminus \bE$ only one of the following holds:
    \begin{enumerate}
        \item $x$ is $\cO$-wim-generated;
        \item $x$ is $\cO$-residually generated;
        \item $x$ is $\cO$-purely valuational.
    \end{enumerate}
    In particular if $x$ is $\cO$-residually generated over $\bE$, then $\bE\langle x \rangle$ does not contain elements that are $\cO$-wim over $\bE$.
\end{fact}

Recall from \cite[Sec.~1]{miller1993growth} that if $\bE$ is an o-minimal expansion of a field, then a \emph{power-function} of $\bE$ is a definable endomorphism of $(\bE^{>0}, \cdot)$; by o-minimality a \emph{power-function} $\theta: \bE^{>0} \to \bE^{>0}$ is always monotone and differentiable, its \emph{exponent} $\lambda$ is then defined as $\theta'(1)$ and $\theta(x)$ is usually written as $x^{\lambda}$.

Also recall that the exponents of (the power-functions of) $\bE$ always form a subfield of $\bE$ and recall from \cite[Thm.~3.6]{miller1993growth} that for every complete o-minimal theory $T$ exactly one of the following holds:
\begin{itemize}
    \item $T$ is \emph{power bounded}: for every model $\bE\models T$, all power functions of $\bE$ are $\emptyset$-definable and all $\bE$-definable total unary functions are eventually bounded by some power-function;
    \item $T$ is \emph{exponential}: for every model $\bE \models T$, there is a $T$-definable ordered group isomorphism $\exp: (\bE, +, <)\to (\bE^{>0}, \cdot, <)$ such that $\exp'(0)=1$.
\end{itemize}

When $T$ is power-bounded a stronger version of Fact~\ref{fact:cuts_tricho} holds: this is known as the \emph{residue-valuation property} or \emph{rv-property} of power-bounded o-minimal theories.

\begin{fact}[{Tyne \cite[Thms.~12.10 and 13.4]{tyne2003t}, van den Dries and Speissegger \cite[9.2 and 10.1]{dries2000field}}]\label{fact:rv-property}
    If $T$ is power bounded with field of exponents $\Lambda$, then for every $x \in \bU \setminus \bE$, only one of the following holds
    \begin{enumerate}
        \item $x$ is $\cO$-wim;
        \item there are $c,d \in \bE^{\neq 0}$ such that $\res(d(x-c)) \notin \res_\cO(\bE)$;
        \item there is $c \in \bE$ such that $\val_\cO(\bE \langle x \rangle) = \val_\cO(\bE) + \Lambda \cdot \val_\cO(x-c)$.
    \end{enumerate}
\end{fact}

\section{Main results}

\subsection{Res-constructible extensions}
We recall the definition of res-constructible extension from the introduction and give some of its basic properties.

\begin{definition}\label{def:res_constructible}
	Let $\bE \preceq \bE_*\models T$ and let $\cO_*$ be a non-trivial $T$-convex subring of $\bE_*$. A possibly infinite tuple $\overline{y}\coloneqq (y_i: i\in I)$ in $\cO_*$ will be called a \emph{$\mathcal{O}_*$-res-construction} over $\bE$ if $\big(\res_{\mathcal{O}_*}(y_i): i \in I\big)$ is $\dcl_T$-independent over $\res(\bE, \cO)$. We then call the extension $\bE \langle y_i :i \in I \rangle$ of $\bE$ \emph{$\cO_*$-res-constructible} over $\bE$ (by the res-construction $(y_i: i \in I)$). When $\mathcal{O}_*$ is clear from the context we will just say \emph{res-construction} and \emph{res-constructible}.

 	We call an extension $(\bE,\cO) \preceq (\bE_*, \cO_*) \models T_\convex$ \emph{residually constructible} (\emph{res-constructible} for short) when $\bE_*$ is $\cO_*$-res constructible over $\bE$.
\end{definition}

\begin{remark}\label{rmk:res-constructions_are_independent}
    Notice that if $\overline{y}\coloneqq (y_i: i\in I)$ is a $\cO_*$-res-construction over $\bE$, then it is in particular $\dcl_T$-independent over $\bE$. This is because by Fact~\ref{fact:res_dim_inequality}, for all finite $F\subseteq I$, $\res_{\cO}(\bE \langle y_i:i \in F\rangle)=\res_{\cO}(\bE) \langle \res(y_i): i \in F\rangle$.
\end{remark}

\begin{remark}\label{rmk:res-constructions_transitive}
    Res-constructible extension are transitive (i.e.\ closed under composition): if $(\bE_0 , \cO_0) \preceq (\bE_1, \cO_1)$ and $(\bE_1, \cO_1)\preceq (\bE_2, \cO_2)$ are res-constructible, then $(\bE_0, \cO_0) \preceq (\bE_2, \cO_2)$ is res-constructible as well.
\end{remark}

\begin{remark}\label{rmk:res-constructibility_and_basis_orderings}
    In the same setting of Definition~\ref{def:res_constructible}, let $\lambda$ be an ordinal, and $(y_i: i <\lambda)$ a $\lambda$-indexed sequence in $\cO_*$. Then the following are equivalent:
    \begin{enumerate}
        \item $(y_i: i<\lambda)$ is a $\mathcal{O}_*$-res-construction over $\bE$;
        \item for all $i<\lambda$, we have $\res_{\cO_*}(y_i) \in \res(\bE \langle y_i: j < i+1\rangle) \setminus \res_{\cO_*}(\bE \langle y_j: j< i\rangle)$.
    \end{enumerate}
    In particular, by Fact~\ref{fact:res_dim_inequality}, for an extension $(\bE, \cO) \preceq (\bE_*, \cO_*)$ to be res-constructible, it is enough that for some $\dcl_T$-basis $B$ of $\bE_*$ over $\bE$, there is an ordinal indexing $(b_i:i <\alpha)$ of $B$ such that for all $i<\alpha$, one has $\res(\bE\langle b_j: j < i \rangle) \neq \res(\bE \langle b_j: j < i+1 \rangle)$.
\end{remark}

\begin{lemma}\label{lem:res-constructible_residue_sections}
    If $(\bE, \cO)\preceq (\bE_*, \cO_*)$ is res-constructible by the ordinal indexed res-construction $(y_i: i < \lambda)$, and $\bK$ is an elementary residue section of $(\bE, \cO)$, then $\bK \langle y_i: i< \lambda\rangle$ is an elementary residue section of $(\bE_*, \cO_*)$.
    \begin{proof}
        Notice that by Fact~\ref{fact:res_iso_max}~(3) it suffices to show that $\res(\bK \langle y_i : i <\lambda \rangle)=\res(\bE_*, \cO_*)$.
        Proceed by induction on $\lambda$. Let $\lambda= \beta+1$ and suppose the thesis holds for $\beta$, so assume that $\bK \langle y_i: i< \lambda\rangle$ is an elementary a residue section of $(\bE\langle y_i: i<\beta\rangle, \cO_* \cap \bE\langle y_i: i<\beta\rangle)$. Since by Fact~\ref{fact:res_dim_inequality} we must then have $\res(\bE \langle y_i: j < \lambda\rangle) = \res(\bE \langle y_i: j < \beta\rangle) \langle \res(y_\beta)\rangle$, we deduce by Fact~\ref{fact:res_iso_max} that $\res(\bK \langle y_i : i <\lambda \rangle)=\res(\bE \langle y_i: j < \lambda\rangle)$.

        Suppose instead that $\lambda$ is a limit ordinal and the thesis holds for each $\beta<\lambda$. If $x \in \bE_*$, then since $\lambda$ is a limit ordinal, there is $\beta<\lambda$ such that $x \in \bE \langle y_i : i< \beta \rangle$ and hence by inductive hypothesis $\res(x) \in \res(\bK \langle y_i : i <\beta \rangle) \subseteq \res(\bK \langle y_i:i < \lambda \rangle)$.
    \end{proof}
\end{lemma}

The isomorphism type of a res-constructible extension only depends on the isomorphism type of the residue field extension $\res(\bE_*, \cO_*) \succeq \res(\bE, \cO)$.

\begin{lemma}\label{lem:res_constructible_ext_iso_type}
    If $(\bE, \cO) \preceq (\bE_1, \cO_1)$ and $(\bE, \cO) \preceq (\bE_2, \cO_2)$ are res-constructible extensions and $\varphi: \res(\bE_1, \cO_1) \to \res(\bE_2, \cO_2)$ is an isomorphism of models of $T$ over $\res(\bE, \cO)$, then there is an isomorphism $\psi: (\bE_1, \cO_1) \to (\bE_2, \cO_2)$ over $\bE$ such that $\res(\psi)=\varphi$.
    \begin{proof}
        Let $\bK$ be an elementary residue-section of $(\bE, \cO)$ and $(y_i: i< \lambda)$, $(z_i: i< \mu)$ be res-construction for $\bE_1$ and $\bE_2$. Thus $\bE= \bK \langle \overline{d} \rangle$ where $\overline{d}$ is a tuple such that $\tp(\overline{d}/\bK)$ is definable.
        By Lemma~\ref{lem:res-constructible_residue_sections}, $\bK_1\coloneqq \bK \langle y_i: i< \lambda \rangle$ and $\bK_2\coloneqq \bK \langle z_i : i < \mu\rangle$ are elementary residue sections for $(\bE_1,\cO_1)$ and $(\bE_2, \cO_2)$ respectively and moreover $\bE_1 = \bK_1 \langle \overline{d}\rangle$ and $\bE_2 = \bK_2 \langle \overline{d} \rangle$.
        Now observe that $\varphi$ induces an isomorphism $h: \bK_1 \to \bK_2$ over $\bK$. Finally notice that for both $i=1$ and $i=2$, $\tp(\overline{d}/\bK_i)$ is the unique definable extension of $\tp(\overline{d}/\bK)$ to $\bK_i$, in particular $h_* \tp(\overline{d}/\bK_1)=\tp(\overline{d}/\bK_2)$. Thus $h$ extends to the sought isomorphism $\psi: (\bE_1, \cO_1) \to (\bE_2, \cO_2)$ over $\bE$. 
    \end{proof}
\end{lemma}

Res-constructible extension only contain residually generated elements.

\begin{lemma}\label{lem:res_implies_no_weak_elements}
    If $(\bE, \cO) \preceq (\bE_*, \cO_*)$ is a res-constructible extension then every $x \in \bE_* \setminus \bE$ is $\cO_*$-residually generated over $\bE$.
    \begin{proof}
        Let $(r_i: i< \alpha)$ be an ordinal-indexed res-construction for $\bE_*$ over $\bE$. Let $j$ be minimal such that $\tp(x/\bE)$ is realized in $\bE \langle r_i: i < j+1 \rangle$. Thus there is $x' \in \bE \langle r_i: i < j+1 \rangle \setminus \bE \langle r_i: i<j \rangle$ such that $\tp(x'/\bE)=\tp(x/\bE)$. Now let $F \subseteq j$ be such that $x' \in \bE \langle r_i : i \in F\cup\{j\}\rangle$. By Fact~\ref{fact:res_dim_inequality} we then have
        \[\res(\bE\langle x ', r_i: i \in F\rangle) = \res(\bE\langle r_i: i \in F \rangle)\langle \res(r_j) \rangle.\]
        In particular $\dim_{\dcl} (\res(\bE\langle  x ', r_i: i \in F\rangle)/\res(\bE))=|F|+1$.
        Suppose toward contradiction that $x'$ was not residually generated over $\bE$, i.e.\ that $\res(\bE\langle x '\rangle) = \res (\bE)$, then, again by Fact~\ref{fact:res_dim_inequality}, we would have $\dim_{\dcl} (\res(\bE\langle x', r_i :i<F\rangle/\res(\bE))\le |F|$, contradiction. Thus $x'$ is residually generated over $\bE$ and so also $x$ is, because $\tp(x/\bE)=\tp(x'/\bE)$.
    \end{proof}
\end{lemma}

\begin{corollary}\label{cor:finite_dim_fact}
    If $(\bE, \cO) \prec (\bE_*, \cO_*)$ is a res-constructible extension and $\overline{y}$ is a finite tuple, then $\bE_*$ is $\cO_*$-res-constructible over $\bE \langle \overline{y} \rangle$ and $\bE \langle \overline{y} \rangle$ is $\cO_*$-res-constructible over $\bE$.
    \begin{proof}
        It suffices to prove the case in which $\overline{y}=(y_0)$ is a single element of $\bE_*$. If $y_0 \notin \bE$, then by Lemma~\ref{lem:res_implies_no_weak_elements}, $\bE \langle y_0\rangle$ is res-constructible over $\bE$, say by some $r \in \bE \langle y_0 \rangle$, such that $\res(r) \notin \res(\bE)$. Let $\bK$ be an elementary residue section of $(\bE, \cO)$ and let $\bK_*$ an extension of $\bK$ to an elementary residue section of $\bE_*$ such that $\bE_*=\dcl(\bE \cup \bK_*)$, which exists by Lemma~\ref{lem:res-constructible_residue_sections}.
        Now if $\overline{s}$ is a $\dcl_T$-basis of $\bK_*$ over $\bK\langle\std_{\bK_*}(r)\rangle$, then $\overline{s}$ is a $\cO_*$-res-construction for $\bE_*$ over $\bE\langle r \rangle=\bE\langle y_0\rangle$.
        In fact $\res(r, \overline{s})=\res(\std_{\bK_*}(r), \overline{s})$ is $\dcl$-independent over $\res(\bE, \cO)$; and since then $r\notin \bE \langle \overline{s}\rangle$, we also have by the exchange property that $\bE \langle r, \overline{s}\rangle =\bE_*$. 
    \end{proof}
\end{corollary}

\begin{lemma}\label{lem:countable_residual_cuts}
    Let $(\bE, \cO)\models T_\convex$, and let $x$ be an element of an elementary extension $(\bE_*,\cO_*)$, such that $x$ is not residually generated over $(\bE, \cO)$. If $\val(\bE, \cO)$ does not contain any uncountable ordinal, then there is a countable $S\subseteq \bE$, such that $\tp_{L_\convex}(x/S) \vdash \tp_{L_\convex}(x/\bE)$.
    \begin{proof}
        To show that an $x\notin \bE$ is such that $\tp_{L_\convex}(x/S) \vdash \tp_{L_\convex}(x/\bE)$ for some countable $S\subseteq \bE$, it is enough to show that each of $\bE^{<x}$ and $\bE^{<-x}$ either has countable cofinality, or is of the form $\{c\pm y : y \in \bE, y\leq 0 \text{ or }\val(y)\geq v\}$ for some $c \in \bE$, $v \in \val(\bE)$, and some choice of $\pm$. By symmetry, it suffices to check $\bE^{<x}$.
        Now, if $\val(x-\bE^{<x})$ has no maximum, then it is included in $\val(\bE)$ and its cofinality equals the cofinality of $\bE^{<x}$, whence we are done by hypothesis. Suppose therefore that $\val(x-\bE^{<x})$ has a maximum, and let it be achieved by $c \in \bE^{<x}$. Note that since $x$ is not residually generated, $\val(x-c) \notin \val(\bE)$.
        Consider $\{ \val(y) : y \in \bE, 0<y<x-c\}$. If this has no minimum, then it has countable cofinality, so $\bE^{<x} = c+\bE^{<x-c}$ has countable cofinality.
        If it has a minimum, say $v$, then, since $v \in \val(\bE)$, $v > \val(x-c)$.
        Thus, for all $y \in \bE$, we have that $c+y<x$ iff $y\leq 0$ or $\val(y)\geq v$, which is precisely the second possibility.
    \end{proof}
\end{lemma}

In Section~\ref{sec:Tressl_problem} we will need the following two easy facts about res-constructible extensions with respect to different $T$-convex valuations.

\begin{lemma}\label{lem:res-constr_for_different_vals}
    Let $\bE \preceq \bU \models T$ and let $\cO \subseteq \cO'$ be $T$-convex valuation subrings of $\bU$. If $\overline{x}$ is a $\cO$-res-construction over $\bE$ then $\overline{x}$ is also an $\cO'$-res-construction over $\bE$ and $\res_{\cO'}(\overline{x})$ is a $\res_{\cO'}(\cO)$-res construction over $\res_\cO(\bE)$.
    \begin{proof}
        Notice that $V\coloneqq \res_{\cO'}(\cO)$ is a $T$-convex valuation subring of $\res_{\cO'}(\bU)$ and that $\res_{\cO}$ factors through $\res_{\cO'}$ as $\res_{\cO} = \res_{V} \circ \res_{\cO'}$.
        Now suppose that $\overline{x}\coloneqq (x_i: i \in I)$ is such that $(\res_\cO(x_i): i \in I)$ is $\dcl$-independent over $\res_\cO(\bE)$. It then follows from Remark~\ref{rmk:res-constructions_are_independent} that $\res_{\cO'}(\overline{x})$ is independent over $\res_{\cO'}(\bE)$ because $\res_{\cO}(\overline{x})=\res_V (\res_{\cO'}(\overline{x}))$ is $\dcl_T$-independent.
    \end{proof}
\end{lemma}

\begin{corollary}\label{cor:res-gen_for_different_vals}
    Let $\bE \prec \bU \models T$, $\cO \subseteq \cO'$ be $T$-convex valuation subrings of $\bU$, and $y \in \bU \setminus \bE$. If $y$ is $\cO$-residually generated, then it is $\cO'$-residually generated.
\end{corollary}

\subsection{Right factors of res-constructible extensions} This subsection is devoted to the proof of Theorem~\ref{introthm:Main}, which will result as a wrap-up of Theorems~\ref{thm:mainthm} and \ref{thm:mainthm_negative}. Throughout this subsection we fix a res-constructible extension $(\bE_*, \cO_*) \succeq (\bE, \cO)$ of model of $T_\convex$ and let $\overline{r}$ be a $\cO_*$-res-construction for $\bE_*$ over $\bE$. We also let $\bE_1$ be an intermediate extension $\bE \preceq \bE_1 \preceq \bE_*$.

\begin{definition}\label{def:strong}
    Let $\bK\preceq_\tame \bE$ be a residue section for $(\bE, \cO)$ and $\bK_* \coloneqq  \bK \langle \overline{r} \rangle$, thus $\bK_*$ is an elementary residue section for $(\bE_*, \cO_*)$ by Lemma~\ref{lem:res-constructible_residue_sections}.
    We will say that a $\cO_*$-res-construction $\overline{x}$ in $\bE_1$ over $\bE$ is
    \begin{enumerate}
        \item \emph{$\bK_*$-orthogonal} in $\bE_1$ if $\tp_{L_\convex}(\std_{\bK_*}(\overline{x}) /\bE\langle \overline{x}\rangle) \vdash \tp_{L_\convex}(\std_{\bK_*}(\overline{x}) /\bE_1)$;
        \item \emph{$\bK_*$-strong} in $\bE_1$ (notation $\overline{x} \triangleleft_{\bK_*} \bE_1$) if it is $\bK_*$-orthogonal in $\bE_1$, and $\bE \langle \overline{x}\rangle = \bE_1 \cap \bE \langle \std_{\bK^*}(\overline{x})\rangle$.
    \end{enumerate}
\end{definition}

\begin{lemma}\label{lem:invariance}
    Let $\overline{x} \triangleleft_{\bK_*}\bE_1$, and suppose $\psi: (\bE_*, \cO_*)\to (\bE_*, \cO_*)$ is an elementary self-embedding fixing $\std_{\bK_*}(\overline{x})$ and $\bE$ elementwise. Then $\psi(\overline{x}) \triangleleft_{\bK_*}\psi(\bE_1)$.
    \begin{proof}
        Since $\overline{x} \triangleleft_{\bK_*}\bE_1$, then $\overline{x} \subset \bE\langle \std_{\bK_*}(\overline{x})\rangle$. Thus, $\psi$ also fixes $\overline{x}$ elementwise. This implies that $\psi(\std_{\bK_*}(\overline{x}))=\std_{\bK_*}(\psi (\overline{x}))$. Then, by applying the elementarity of $\psi$ and the fact that $\overline{x}$ is $\bK_*$-orthogonal in $\bE_1$, we deduce that
        \[\tp\big(\psi(\std_{\bK_*}(\overline{x})) / \psi(\bE)\langle \psi(\overline{x})\rangle\big) \vdash \tp\big(\psi(\std_{\bK_*}(\overline{x})) /\psi(\bE_1)\big).\]
        Applying $\psi(\std_{\bK_*}(\overline{x}))=\std_{\bK_*}(\psi (\overline{x}))$ yields that $\psi(\overline{x})$ is $\bK_*$-orthogonal in $\psi(\bE_1)$. Finally observe that $\bE\langle \psi(\overline{x}) \rangle= \psi(\bE \langle \overline{x} \rangle) = \psi(\bE_1) \cap \psi(\bE \langle \std_{\bK_*}(\overline{x})\rangle) = \psi(\bE_1) \cap \bE \langle \std_{\bK_*}(\psi(\overline{x})) \rangle$.
    \end{proof}
\end{lemma}

\begin{lemma}\label{lem:K_*orthogonal}
    Let $\overline{x}$ be a $\bK_*$-orthogonal $\cO_*$-res-construction in $\bE_1$ and $\overline{y}$ a $\cO_*$-res-construction in $\bE_1$ over $\bE \langle \overline{x} \rangle$. If $z\in \bE_1$ is $\cO_*$-residually generated over $\bE \langle \std_{\bK_*}(\overline{x}), \overline{y} \rangle$ then it is $\cO_*$-residually generated over $\bE \langle \overline{x}, \overline{y} \rangle$.
    \begin{proof}
        Notice that $\res(\bE \langle \std_{\bK_*}(\overline{x}), \overline{y}\rangle) = \res(\bE \langle \overline{x},\overline{y}\rangle)\subseteq \res(\bE_1)$.
        Suppose that $z$ is residually generated over $\bE \langle \std_{\bK_*}(\overline{x}), \overline{y}\rangle$. Then there is a finite subtuple $\overline{s}$ of $\std_{\bK_*}(\overline{x})$ and a $\bE\langle \overline{y} \rangle$-definable function $f$ such that $\res(f(\overline{s}, z)) \notin \res(\bE \langle \std_{\bK_*}(\overline{x}), \overline{y}\rangle)$.
        However $\res(f(\overline{s},z))=\res(a)$ for some $a \in \bE_1$ and by $\bK_*$-orthogonality of $\overline{x}$ in $\bE_1$, we have $\tp_{L_\convex}(s/\bE\langle \overline{x} \rangle) \vdash \res(f(\overline{s},z))=\res(a)$.
        But then there is a $L_\convex$-formula $\varphi(\overline{s})$ with parameters from $\bE \langle \overline{x}\rangle$ such that $\varphi(\overline{s})\vdash \res(f(\overline{s},z))=\res(a)$. Taking $\overline{s}'$ in $\bE \langle \overline{x} \rangle$ such that $\big(\bE \langle \overline{x} \rangle, \cO_* \cap \bE \langle \overline{x}\rangle\big) \models \varphi(\overline{s}')$ we thus get $\res(f(\overline{s}',z))=\res(a)=\res(f(\overline{s},z))$, whence $z$ is $\cO_*$-residually generated over $\bE \langle \overline{x}, \overline{y}\rangle$.
    \end{proof}
\end{lemma}

\begin{lemma}\label{lem:toname}
    Let $\overline{x} \triangleleft_{\bK_*} \bE_1$ and let $\overline{y}$ be a finite res-construction in $\bE_1$ over $\bE \langle \overline{x} \rangle$. Then every $z \in \bE_1 \setminus \bE \langle \overline{x}, \overline{y}\rangle$ is $\cO_*$-residually generated over $\bE \langle \overline{x}, \overline{y}\rangle$.
    \begin{proof}
        Let $\overline{r}'$ be a completion of $\std_{\bK_*}(\overline{x},\overline{y})$ to a $\dcl_T$-basis of $\bK_*$. Set $\psi(\std_{\bK_*}(\overline{y}))\coloneqq \overline{y}$, $\psi(\std_{\bK_*}(\overline{x}))\coloneqq \std_{\bK_*}(\overline{x})$ and $\psi_{\bK_*}(\overline{r}')=\overline{r}'$.
        Notice that $\psi$ extends to an elementary self-embedding of $(\bE_*, \cO_*)$ over $\bE$ and that it fixes $\bE\langle \std_{\bK_*}(\overline{x}), \overline{r}' \rangle$ which has finite codimension in $\bE_*$, therefore $\psi$ is an automorphism. By Lemma~\ref{lem:invariance}, we have $\psi^{-1}(\overline{x}) \triangleleft_{\psi^{-1}(\bK_*)} \psi^{-1}(\bE_1)$. Moreover $\psi^{-1}(z) \in \psi^{-1}(\bE_1)$ must be residual over $\psi^{-1}(\bE\langle \std_{\bK_*}(\overline{x}), \overline{y} \rangle)=\bE \langle \std_{\bK_*}(\overline{x}), \std_{\bK_*}(\overline{y})\rangle$ whence by Lemma~\ref{lem:K_*orthogonal} it is $\cO_*$-residually generated over $\psi^{-1}(\bE \langle \overline{x}, \overline{y}\rangle)$. Applying $\psi$ finally yields that $z$ is $\cO_*$-residual over $\bE\langle \overline{x}, \overline{y}\rangle$.
    \end{proof}
\end{lemma}

As a Corollary of Lemma~\ref{lem:toname}, we get that if $\overline{x}\triangleleft_{\bK_*} \bE_1$, then for all $\bE_0$ such that $\bE\langle \overline{x}\rangle \preceq \bE_0\preceq \bE_1$ and $\dim_{\dcl}(\bE_0/\bE\langle \overline{x}\rangle)\le \aleph_0$, we have that $\bE_0$ is $\cO_*$-res-constructible over $\bE\langle \overline{x} \rangle$. We give a more precise statement of this.

\begin{corollary}\label{cor:toname1}
    Suppose that $\overline{x}\trianglelefteq_{\bK_*} \bE_1$ and $\overline{y}\coloneqq (y_i: i< \alpha)$ in $\bE_1$ is $\dcl_T$-independent over $\bE \langle \overline{x}\rangle$ with $\alpha \le \omega$. Then there is a $\cO_*$-res-construction $\overline{z}\coloneqq (z_i : i < \alpha)$ in $\bE \langle \overline{x}, \overline{y}\rangle$ over $\bE \langle \overline{x}\rangle$, such that for all $n\le \alpha$, $\bE \langle \overline{x}, x_i : i <n \rangle = \bE \langle \overline{x}, y_i: i < n\rangle$. 
    \begin{proof}
        By induction, using  Lemma~\ref{lem:toname}, we can build a $\cO_*$-res-construction $(z_i: i< \alpha)$ over $\bE \langle \overline{x}\rangle$ such that for all $n<\alpha$, we have $\bE \langle \overline{x} , z_i: i<n\rangle = \bE \langle \overline{x}, y_i: i<n\rangle$.
    \end{proof}
\end{corollary}

$\bK_*$-strong res-constructions are closed under directed unions.

\begin{lemma}\label{lem:strong_dir_unions}
    Let $\overline{x}\coloneqq (x_i : i \in I)$ be family of elements in $\cO_* \cap \bE_1$. If for all finite $F \subseteq I$, $\overline{x}_{F}\coloneqq (x_i : i\in F)\triangleleft_{\bK_*}\bE_1$, then $\overline{x} \triangleleft_{\bK_*} \bE_1$.
    \begin{proof}
        The fact that $\overline{x}$ is a $\cO_*$-res/construction follows from the finite character of $\dcl_T$-independence. Similarly, that $\overline{x}$ is $\bK_*$-orthogonal follows from the fact that every formula in $\tp(\std_{\bK_*}(\overline{x})/\bE_1)$ is in fact in some $\tp(\std_{\bK_*}(\overline{x}_{F})/\bE_1)$ for some finite $F \subseteq I$.
        Finally $\bE \langle \overline{x} \rangle = \bigcup_{F} \bE \langle \overline{x}_F \rangle = \bE_1 \cap \bigcup_{F} \bE \langle \std_{\bK_*}(\overline{x})\rangle$, where in the unions $F$ ranges among the finite subsets of $I$.
    \end{proof}
\end{lemma}

\begin{lemma}
    Suppose that $\val_{\cO_*}(\bE_1)$ contains no uncountable well ordered set and $\std_{\bK_*}(\bE_1)=\bK_*$. Let $\bE \preceq \bE_0 \preceq \bE_1$ and $\overline{y}$ be a finite $\cO_*$-res-construction in $\bE_1$ over $\bE_0$.
    Then there is a countable tuple $\overline{z}$ in $\bE_1$ such that $\tp_{L_{\convex}}(\std_{\bK_*}(\overline{y})/\bE_0\langle \overline{z} \rangle) \vdash \tp_{L_{\convex}}(\std_{\bK_*}(\overline{y})/\bE_1)$.
    \begin{proof}
        Notice that we can reduce to the case in which $y$ is a single residual element.
        Let $s=\std_{\bK_*}(y)$. We can assume that $s \notin \bE_1$ for otherwise the thesis is trivial. 
        Since $\std_{\bK_*}(\bE_1)=\bK_*$, $s$ is not residually generated over $\bE_1$, so the thesis follows from the fact that $\val_{\cO_*}(\bE_1)$ contains no uncountable well ordered subset and from Lemma~\ref{lem:countable_residual_cuts}.
    \end{proof}
\end{lemma}

\begin{proposition}\label{prop:extending_strong}
    Let $\overline{x}\triangleleft_{\bK_*} \bE_1$ and let $B$ be a $\dcl_T$-basis of $\bE_1$ over $\bE \langle \overline{x}\rangle$. Suppose that $\val_{\cO_*}(\bE_1)$ contains no uncountable ordinal and that $\std_{\bK_*}(\bE_1)=\bK_*$. If $B$ is infinite, then every $b_0 \in B$ extends to a countable sequence $(b_i: i <\omega)$ in $B$ such that there is $\overline{x}'\coloneqq (x'_i: i < \omega)$ with $(\overline{x},\overline{x}') \trianglelefteq_{\bK_*} \bE_1$ and $\bE \langle \overline{x} ,x_i' :i<n\rangle=\bE \langle \overline{x}, b_i: i<n\rangle$ for all $n\le \omega$.
    \begin{proof}
        Consider functions $f, g: [B]^{<\aleph_0} \to [B]^{<\aleph_1}$ defined as follows. 
        \begin{itemize}
            \item given a finite subset $S\subseteq B$, let $S'$ be a $\dcl_T$-basis of $\std_{\bK_*}(\bE \langle \overline{x}, S\rangle)$ over $\std_{\bK_*}(\bE \langle \overline{x} \rangle)$. Set $f(S)$ to be a countable subset of $B$ such that
            \[\tp_{L_{\convex}}(S'/\bE\langle \overline{x}, f(S) \rangle)\vdash \tp_{L_{\convex}}(S'/\bE_1);\]
            \item given a finite subset $S\subseteq B$, choose a finite $R\subseteq \bK_*$ such that $S\subseteq \bE\langle R\rangle$, then set $g(S)$ to be such that $R \subseteq \std_{\bK_*}(\bE \langle \overline{x}, g(S)\rangle)$ and $\bE \langle \std_{\bK_*}(\overline{x}) , R\rangle \cap \bE_1 \subseteq \bE \langle \overline{x}, g(S) \rangle$ (notice that one can choose a finite such $g(S)$).
        \end{itemize}
        By a standard diagonalization argument, given $b_0 \in B$ there is a countable subset $\{b_i: i < \omega\}$ of $B$ closed under $f$ and $g$. Let $\overline{x}'$ be a res-construction of $\bE \langle \overline{x}, b_i: i< \omega\rangle$ over $\bE \langle \overline{x} \rangle$. To see it is $\bK_*$-orthogonal notice that $\tp(\std_{\bK_*}(\overline{x}')/\bE \langle \overline{x}, \overline{x}' \rangle)\vdash \tp(\std_{\bK_*}(\overline{x}')/\bE_1)$ by compactness and closure under $f$. To see it is $\bK_*$-strong, notice that $\bE \langle \overline{x}, \overline{x}' \rangle \subseteq \bE \langle \std_{\bK_*}(\overline{x}), \std_{\bK_*}(\overline{x}')\rangle$ and $\bE \langle \std_{\bK_*}(\overline{x}), \std_{\bK_*}(\overline{x}')\rangle \cap \bE_1 \subseteq \bE \langle \overline{x}, \overline{x}' \rangle$ by closure under $g$.
    \end{proof}
\end{proposition}

An immediate Corollary of Proposition~\ref{prop:extending_strong}, is the following, which amounts to the first part of our main theorem, but with extra assumption on $\bE_1$ that $\std_{\bK_*}(\bE_1)=\bK_*$.

\begin{corollary}\label{cor:mainthm_with_extra_hyp}
    Suppose that $\overline{x} \triangleleft_{\bK_*} \bE_1$, $\val_{\cO_*}(\bE_1)$ contains no uncountable ordinal, and $\std_{\bK_*}(\bE_1)=\bK_*$. Then for every $\dcl_T$-basis $B$ of $\bE_1$ over $\bE\langle \overline{x} \rangle$, there is an ordinal-indexing $(b_i: i <\lambda)$ of $B$, such that for all $i<\lambda$, $\res(\bE \langle \overline{x}, b_j: j<i\rangle) \neq \res(\bE \langle \overline{x}, b_j: j\le i\rangle)$.
    \begin{proof}
        Using Proposition~\ref{prop:extending_strong} and Lemma~\ref{lem:strong_dir_unions} one can find some ordinal $\alpha$, $\overline{x}'\coloneqq (x_i': i< \alpha)$ in $\bE_1$ and $(b_i: i< \alpha)$ in $B$, such that $\{b_i: i <\alpha\}$ is cofinite in $B$, $(\overline{x}, \overline{x}')\triangleleft_{\bK_*} \bE_1$, and for all $i\le\alpha$, $\bE \langle \overline{x}, x_j':j<i\rangle = \bE \langle \overline{x}, b_j:j<i\rangle$. To conclude it suffices to invoke Corollary~\ref{cor:toname1} with $\overline{y}$ an enumeration of the finite set $B \setminus \{b_i : i< \alpha\}$.
    \end{proof}
\end{corollary}

We will use the following Lemma to remove the extra assumption from the above Corollary.

\begin{lemma}\label{lem:toname2}
    Let $(b_i : i <\alpha)$ be an ordinal-indexed sequence of elements of $\bE_*$.
    Let $\overline{r}$ be a res-construction over $\bE \langle b_i : i <\alpha\rangle$ and suppose that for all $\beta<\alpha$, $\res(\bE \langle \overline{r}, b_i: i< \beta \rangle)\neq \res(\bE \langle \overline{r}, b_i: i\le \beta\rangle)$. Then for all $\beta<\alpha$, $\res(\bE \langle b_i: i< \beta \rangle)\neq \res(\bE \langle b_i: i\le \beta\rangle)$. In particular $\bE \langle b_i: i< \alpha\rangle$ is $\cO_*$-res-constructible over $\bE$.
    \begin{proof}
        Note that if the hypotheses hold for $\alpha$, they also hold for any $\alpha' < \alpha$. Therefore, by induction, we may assume the statement holds for all $\alpha'<\alpha$.
        If $\alpha$ is a limit ordinal, then there is nothing to show.
        Suppose thus that $\alpha=\alpha'+1$. 
        Notice that $\overline{r}$ is in particular also an $\cO_*$-res-construction over $\bE \langle b_i: i< \alpha'\rangle$, so by the inductive hypothesis $\res(\bE \langle b_i: i<\beta\rangle) \neq \res(\bE \langle b_i: i \le \beta \rangle)$ for all $\beta<\alpha'$ and we only need to show that $\res(\bE \langle b_i: i <\alpha' \rangle) \neq \res(\bE \langle b_i: i\leq \alpha'\rangle)$, i.e.\ that $b_{\alpha'}$ is $\cO_*$-residually generated over $\bE \langle b_i: i <\alpha' \rangle$.
        Notice that $b_{\alpha'}$ is $\cO_*$-residually generated over $\bE \langle \overline{r}, b_i: i<\alpha'\rangle$ by hypothesis, so there is a finite $\overline{r}'\subseteq \overline{r}$, a finite $F \subseteq \alpha'$, and an $\bE$-definable function $f$, such that $\res(f(\overline{r}', b_F)) \notin \res(\bE \langle \overline{r}, b_i: i <\alpha'\rangle)$.
        Since then $\dim_{\dcl}\big( \res(\bE \langle \overline{r}', b_i: i <\alpha' \rangle)/\res(\bE\langle b_i: i< \alpha'\rangle)\big) = |\overline{r}'|$ and $\dim_{\dcl}\big( \res(\bE \langle \overline{r}', b_i: i \leq\alpha' \rangle)/\res(\bE\langle b_i: i< \alpha'\rangle)\big) = |\overline{r}'|+1$ it follows by a dimension argument that $\res(\bE \langle b_i :i< \alpha'\rangle)\neq \res(\bE \langle b_i: i <\alpha\rangle)$. 
    \end{proof}
\end{lemma}

\begin{theorem}\label{thm:mainthm}
    Suppose $\bE_*\succ \bE$ is $\cO_*$-res-constructible over $\bE$, $\val_{\cO_*}(\bE_*)$ contains no uncountable well-order and $\bE_* \succ \bE_1 \succ \bE$. Then $\bE_1$ is $\cO_*$-res-constructible.
    \begin{proof}
        Let $S$ be a complement of $\std_{\bK_*}(\bE_1)$ to a dcl-basis of $\bK_*$. Set $\bE_2\coloneqq \bE_1\langle S\rangle$, $\bE_0\coloneqq \bE \langle S\rangle$ and let $B$ be a $\dcl$-basis of $\bE_1$ over $\bE$. 
        Notice that $\bE_*$ is res-constructible over $\bE_0$, $\bE_0 \prec \bE_2 \prec \bE_*$, $\bE_2=\bE_0 \langle B\rangle$, and $\std_{\bK_*}(\bE_2)=\bK_*$.
        By Corollary~\ref{cor:mainthm_with_extra_hyp}, we can find an ordinal-indexed enumeration $(b_i: i< \alpha)$ of $B$ such that $\res(\bE_0 \langle b_i: i< \beta \rangle)\neq \res(\bE_0 \langle b_i: i\le \beta\rangle)$ for all $\beta <\alpha$.
        By Lemma~\ref{lem:toname2}, it follows that $\res(\bE \langle b_i: i< \beta \rangle)\neq \res(\bE \langle b_i: i\le \beta\rangle)$ for all $\beta <\alpha$. Thus one easily sees that $\bE_1 = \bE \langle B \rangle$ is $\cO_*$-res-constructible over $\bE$.
    \end{proof}
\end{theorem}

\begin{theorem}\label{thm:mainthm_negative}
     Suppose that $(\bE, \cO) \prec (\bE_*, \cO_*)$ is res-constructible and such that:
     \begin{enumerate}
         \item $\val(\bE_*)$ contains an uncountable well ordered subset;
         \item $\bE_*$ has uncountable $\dcl_T$-dimension over $\bE$. 
     \end{enumerate}
     Then there is $\bE \prec \bE_1 \prec \bE_*$ such that $\bE_1$ is not $\cO_*$-res-constructible and moreover $\res(\bE_1)=\res(\bE_*)$.
    \begin{proof}
        By the hypothesis we can find $(s_i: i < \aleph_1)$, $r$ and $(r_i : i <\aleph_1)$ in $\bE_*$ such that $\res(r)$, $\res(r_i)_{i<\aleph_1}$, $\res(s_i)_{i< \kappa}$ are a $\dcl_T$-basis of $\res(\bE_*)$ over $\res(\bE)$ and moreover $\bE_0\coloneqq  \bE \langle s_i : i< \kappa\rangle$ contains a sequence $(x_i: i<\aleph_1)$ such that $\val(x_i)>\val(x_j)$ for all $i < j$.
        Set $\bF_i'\coloneqq  \bE_0 \langle r + r_j x_j: j < i\rangle$. We claim that $\bF\coloneqq \bF'_{\aleph_1}$ is not res-constructible. Suppose toward contradiction that it is: then there is a sequence $(\bF_i)_{i < \aleph_1}$, such that
        \begin{enumerate}
            \item $\bE_0= \bF_0\prec \bF_i\prec \bF_j \prec \bF$ for all $0 \le i<j<\aleph_1$;
            \item $\dim(\bF_{i+1}/\bF_{i})=1= \dim \big( \res(\bF_{i+1})/\res(\bF_i)\big)$ and $\bF_{i} = \bigcup_{j<i}\bF_{j+1}$, for all $i<\aleph_1$.
        \end{enumerate}
        Consider the two functions $f, g: \aleph_1 \to \aleph_1$ defined respectively as
        \[f(\alpha)\coloneqq  \min \{\beta: \bF_\alpha \subseteq \bF'_{\beta}\} \quad \text{and} \quad g(\alpha)\coloneqq \min \{\beta: \bF'_\alpha \subseteq \bF_{\beta}\}.\]
        Notice that these are both increasing and continuous.

        It follows that there is some $0<\delta<\aleph_1$ which is a limit ordinal and a fixed point both of $f$ and $g$. For such $\delta$ we have $\bF_{\delta}=\bF'_\delta$. Since by construction $\bF$ is res-constructible over $\bF_\delta$, by Lemma~\ref{lem:res_implies_no_weak_elements} all elements of $\bF$ should be residually generated over $\bF_\delta=\bF'_\delta$.
        
        Notice that $r+x_\delta r_\delta \in r+x_\beta r_\beta + x_\beta \cO_*$, thus since $r+x_\delta r_\delta$ cannot be weakly immediate over $\bF_\delta$ (by Fact \ref{fact:cuts_tricho}), it follows that there is $a \in \bF_{\delta}$ such that $r+x_\delta r_\delta - a \in \bigcap_{\beta<\delta} x_\beta \co_*$.
        Now there is a finite $F \subseteq \delta$, such that $a \in \bE_0 \langle r+ r_\beta x_\beta: \beta\in F\rangle$. But now for all $\beta \in F$, $\res(r_\beta)=\res((r+r_\beta x_\beta-a)/x_\beta) \in \res (\bE_0 \langle r+r_\beta x_\beta : \beta \in F\rangle)$ and on the other hand $\res(r)=\res(a)=\res(r+r_\beta x_\beta) \in \res(\bE_0\langle r+ r_\beta x_\beta: \beta\in F\rangle)$ as well, thus since $\{\res(r)\} \cup \{\res(r_\beta): \beta \in F\}$ are independent over $\res(\bE_0)$, it would follow that $\res(\bE_0 \langle r+ r_\beta x_\beta: \beta \in F\rangle)$ has dimension $|F|+1$ over $\res(\bE_0)$ contradicting Fact~\ref{fact:res_dim_inequality}.
    \end{proof}
\end{theorem}

\begin{proof}[Proof of Theorem~\ref{introthm:Main}]
    The implication $(1) \Rightarrow (2)$ is Theorem~\ref{thm:mainthm}, $(2) \Rightarrow (3)$ is trivial, and finally $(3) \Rightarrow (1)$ is Theorem~\ref{thm:mainthm_negative}. 
\end{proof}

\section{Application to Tressl's pseudo-completions}\label{sec:Tressl_problem}

We recall the definition of pseudo-completion of \cite{tressl2006pseudo}, and show how our results apply to \cite[Problem~5.12]{tressl2006pseudo}.

\begin{definition}
    We call an elementary extension $\bK \preceq \bK'\models T$, \emph{dense} if $\bK$ is dense in $\bK'$, we call it \emph{cofinal} if $\bK$ is cofinal in $\bK'$.
    Thus a \emph{Cauchy-completion} of some $\bK \models T$, can be defined as an elementary extension $\bK_* \succeq \bK$ which is maximal among the dense elementary extensions $\bK'\succeq \bK$.
\end{definition}

\begin{remark}\label{rmk:Cauchy-completion_terminology}
    In \cite[Sec.~3]{tressl2006pseudo}, what we call Cauchy-completion is called just \emph{completion}. 
\end{remark}

\begin{lemma}\label{lem:dense_in_cofinal}
    Suppose that $\bK \preceq \bK'$ is a dense extension and suppose that $\bK \preceq \bE$ is a cofinal extension such that for all $x \in \bK'$, $\tp(x/\bK)$ is realized in $\bE$. Then there is a unique elementary embedding over $\bK$ of $\bK'$ in $\bE$.
    \begin{proof}
        For the uniqueness, it suffices to show that for all $x \in \bK'\setminus \bK$, $\tp(x/\bK)$ has a unique realization in $\bE$. If it was not the case and $x_0, x_1 \in \bE$, were distinct realizations of $\tp(x/\bK)$ then since $\bK$ is dense in both $\bK\langle x_0\rangle$ and $\bK\langle x_1\rangle$ it would follow that $1/|x_0-x_1|>\bK$ contradicting the cofinality of $\bK$ in $\bE$.
        To prove existence, let $(b_i: i< \alpha)$ be a $\dcl_T$-basis of $\bK'$ over $\bK$. To build an embedding of $\bK$ in to $\bE$ it suffices to show that for all $\beta<\alpha$, if $(c_i: i< \beta)$ is a tuple in $\bE$ with $(b_i:i<\beta)\equiv_\bK (c_i: i< \beta)$, then we can find $c_\beta$ in $\bE$ such that $(b_i:i\le\beta)\equiv_\bK (c_i: i\le \beta)$. This easily follows from the fact that $\tp(b_\beta/\bK)\vdash \tp(b_\beta/\bK \langle b_i: i< \beta\rangle)$ because $\bK'$ is dense in $\bK$.
    \end{proof}
\end{lemma}

\begin{corollary}\label{cor:Cauchy-completion_unique}
    The following hold:
    \begin{enumerate}
        \item if $\bE$ is a cofinal extension of $\bK$, then there is a unique maximal dense extension of $\bK$ in $\bE$;
        \item Cauchy-completions exist and are unique up to a unique isomorphism: if $\bK \preceq \bK'$ and $\bK \preceq \bK''$ are Cauchy-completions of $\bK$, then there is a unique isomorphism $\bK'\simeq \bK''$ over $\bK$.
    \end{enumerate}
\end{corollary}

\begin{definition}\label{def:res-c-dc}
    Let $\bE \preceq \bU \models T$ and $\cO$ a $T$-convex subring of $\bU$. We will call an extension $\bE_*\preceq \bU$ of $\bE$ \emph {$\cO$-residually dense} (resp. \emph{$\cO$-residually cofinal}) if it is $\cO$-res-constructible and moreover the extension $\res_{\cO}(\bE) \preceq \res_{\cO}(\bE_*)$ is dense (resp.\ cofinal).
\end{definition}

\begin{lemma}\label{lem:res_cof}
    Suppose that $\cO\subseteq \cO'$ are $T$-convex valuation subrings of $\bE_1$ and let $\bE_1$ be $\cO$-res-constructible over $\bE$. If $\bE_1$ is $\cO$-residually cofinal over $\bE$, or if $\cO\subsetneq \cO'$, then $\bE_1$ is $\cO'$-residually cofinal over $\bE$.
    \begin{proof}
        We first observe that if $\bK\prec \bK_1\models T$, $V$ is a non-trivial $T$-convex valuation ring of $\bK_1$, and $\bK_1$ is $V$-res-constructible over $\bK$, then $\bK$ is cofinal in $\bK_1$. In fact if there was $t\in \bK_1$ such that $t>\bK$, then $t$ is $V$-purely valuational over $\bK$ so by Fact~\ref{fact:cuts_tricho} $t$ is not $V$-residually generated, contradicting Lemma~\ref{lem:res_implies_no_weak_elements}.
        Now let $\overline{r}$ be an $\cO$-res-construction of $\bE_1$ over $\bE$, then by Lemma~\ref{lem:res-constr_for_different_vals}, it is also a $\cO'$-res construction and $\res_{\cO'}(\overline{r})$ is a $\res_{\cO'}(\cO)$-res-construction for $\res_{\cO'}(\bE_1)$ over $\res_{\cO'}(\bE)$. In particular $\bE_1$ is $\cO'$-res-constructible over $\bE$ and $\bK_1\coloneqq \res_{\cO'}(\bE_1)$ is $V\coloneqq \res_{\cO'}(\cO)$-res-constructible over $\bK\coloneqq \res_{\cO'}(\bE)$. Thus to conclude we only need to show that $\bK$ is cofinal in $\bK_1$.

        If $\cO' \neq \cO$, then $V \neq \res_{\cO'}(\bE_1)$ and we are done by our first observation. 

        On the other hand if $\cO'=\cO$ and $\res_{\cO}(\bE)$ is cofinal in $\res_{\cO}(\bE_1)$, then the thesis is trivial.
    \end{proof}
\end{lemma}

\begin{remark}
    In the situation of the Definition~\ref{def:res-c-dc} above, if $\cO$ and $\cO'$ are $T$-convex subrings of $\bU$ such that $\cO \cap \bE = \cO' \cap \bE$, then $\bE_*\preceq \bU$ is $\cO$-residually dense (resp.\ $\cO$-residually cofinal) over $\bE$ if and only if it is $\cO'$-residually dense (resp.\ $\cO'$-residually cofinal). This is because if $\bE_*$ is $\cO$-residually cofinal over $\bE$, then $\cO \cap \bE_*$ is the only $T$-convex valuation ring of $\bE_*$ whose restriction to $\bE$ is $\bE\cap \cO$.
    
    Thus if $\cO''$ is a $T$-convex valuation subring of $\bE$, we will call an extension $\bE_*$ of $\bE$, $\cO''$-residually dense (resp.\ $\cO''$-residually cofinal) if it is $\cO$-residually dense (resp.\ $\cO$-residually cofinal) for some $T$-convex $\cO$ in $\bU$ with $\cO \cap \bE=\cO''$.
\end{remark}

\begin{lemma}\label{lem:res-dense_ortho_one}
    If $\bE \prec \bE_1$ is a dense extension and $\cO$ is a non-trivial $T$-convex valuation ring of $\bE_1$, then $\bE \not \subseteq \cO$, $\res_\cO(\bE)=\res_\cO(\bE_1)$, and $\val_\cO(\bE)=\val_\cO(\bE_1)$.
    \begin{proof}
        The claim that $\bE \not \subseteq \cO$ is trivial, because if $\bE \subseteq \cO$, then by density of $\bE$ in $\bE_1$ it would follow that $\cO=\bE_1$. 
        
        To prove $\res_\cO(\bE)=\res_\cO(\bE_1)$, suppose toward contradiction that there was an element $x$ of $\bE_1 \setminus \bE$ such that $\res_{\cO}(x)\notin \res_{\cO}(\bE)$. Then there would be no elements of $\bE$ in $x + \co$ whereas $(x+\co)\cap \bE_1$ would be infinite, thus contradicting the density of $\bE$ in $\bE_1$.

        Finally to see that $\val_{\cO}(\bE)=\val_{\cO}(\bE_1)$ notice that since $\bE$ is dense in $\bE_1$, then in particular for all $z \in \bE_1^{>0}$, there is $c \in \bE$ such that $c\geq z$ and $|c-z|/|z|<1$.
    \end{proof}  
\end{lemma}

\begin{remark}\label{rmk:res_constr_lift}
    If $\overline{r}$ is a $\cO$-res-construction over $\bE$ and $\res_\cO(\bE_1)=\res_\cO(\bE)$, then $\overline{r}$ is a $\cO$-res-construction over $\bE_1$.
\end{remark}

\begin{lemma}\label{lem:res-dense_ortho}
    Let $\bE \prec \bU \models T$ and $\cO \subseteq \cO'$ be $T$-convex valuation subrings of $\bU$ such that $\cO \cap \bE \neq \cO' \cap \bE$. Suppose that $\bE_1$ is a $\cO'$-residually dense extension of $\bE$ with $\cO'$-res-construction $\overline{b}$ and $\bE_2$ is a $\cO$-res-constructible extension of $\bE$ with $\cO$-res-construction $\overline{c}$. Then the following hold:
    \begin{enumerate}
        \item $\res_{\cO}(\bE)=\res_\cO(\bE_1)$ and $\val_{\cO}(\bE \cap (\cO'\setminus \co'))= \val_\cO (\bE_1 \cap (\cO'\setminus \co'))$, in particular $\cO \cap \bE_1$ is the unique $T$-convex extension of $\cO \cap \bE$; 
        \item $\overline{c}$ is a $\cO$-res-construction of $\dcl(\bE_2\cup \bE_1)$ over $\bE_1$ and $\dcl(\bE_2\cup \bE_1)$ is $\cO'$-residually cofinal over $\bE$ with $\cO'$-res-construction $(\overline{b}, \overline{c})$;
        \item $\tp_L(\bE_1/\bE) \vdash \tp_L(\bE_1/\bE_2)$.
    \end{enumerate}
    \begin{proof}
        Throughout the proof let $V\coloneqq \res_{\cO'}(\cO)$ so that $\res_\cO= \res_V \circ \res_{\cO'}$.
        
        (1) Since $\bE_1$ is $\cO'$-residually dense, $\res_{\cO'}(\bE_1)$ is a dense extension of $\res_{\cO'}(\bE)$. Thus by Lemma~\ref{lem:res-dense_ortho_one} $\res_V \res_{\cO'}(\bE_1) = \res_V \res_{\cO'}(\bE)$ and
        \[\res_{\cO}(\bE_1)=\res_V \res_{\cO'}(\bE_1) = \res_V \res_{\cO'}(\bE) = \res_{\cO}(\bE).\]
        The statement about the value group is proved in the same way using that $\val_\cO(\bE_1 \cap (\cO' \setminus \co'))=\val_{V}(\res_{\cO'}(\bE_1))$ and Lemma~\ref{lem:res-dense_ortho_one}. The last clause is because in particular $\val_\cO(\bE_1 \cap (\cO' \setminus \co'))=\val_{V}(\res_{\cO'}(\bE_1))$ and $\res_{\cO}(\bE)=\res_\cO(\bE_1)$ imply that there is no $x \in \bE_1$ such that $\cO \cap \bE < x <\bE^{>\cO}$.
        
        (2) The first assertion is by point (1) and Remark~\ref{rmk:res_constr_lift}. As for the second one, notice that by Lemma~\ref{lem:res-constr_for_different_vals} it follows from the first assertion that $(\overline{b}, \overline{c})$ is a $\cO'$-res-construction over $\bE$, because $\overline{b}$ was a $\cO'$-res-construction over $\bE$. Finally, to see that $\res_{\cO'}(\bE)$ is cofinal in $\res_{\cO'}(\dcl(\bE_2 \cup \bE_1))$, recall that $\res_{\cO'}(\bE)$ is dense in $\res_{\cO'}(\bE_1)$ by hypothesis and that $\res_{\cO'}(\bE_1) \prec \res_{\cO'}(\dcl(\bE_2 \cup \bE_1))$ is a $V$-res-constructible extension with $V \not \supseteq \res_{\cO'}(\bE_1)$ and thus in particular it is cofinal by Lemma~\ref{lem:res_cof}.

        (3) Let $\ob$ be a $\cO'$-res-construction of $\bE_1$ over $\bE$.
        Then it suffices to show that for any finite tuple $b'_0, \ldots b'_n \subset \ob$, $\tp(b'_n/\bE\langle b'_0, \ldots b'_{n-1}\rangle) \vdash \tp(b'_n/\bE_2\langle b'_0, \ldots b'_{n-1}\rangle)$.
        
        For this notice that by Remark~\ref{rmk:res_constr_lift} and point (1), we have that $\bE_2\langle b_0', \ldots, b_{n-1}'\rangle$ is $\cO$-res-constructible over $\bE\langle b_0', \ldots b_{n-1}'\rangle$. Thus by Lemma~\ref{lem:res_implies_no_weak_elements}, every element of $\bE_2\langle b_0', \ldots, b_{n-1}'\rangle$ is $\cO$-residually generated over $\bE \langle b_0', \ldots, b_{n-1}'\rangle$.

        Now by point (1), $\res_\cO(\bE \langle b_0', \ldots, b_n'\rangle)=\res_\cO(\bE \langle b_0', \ldots, b_{n-1}'\rangle)=\res_\cO(\bE)$, thus $b_n'$ is not $\cO$-residually generated over $\bE \langle b_0', \ldots, b_{n-1}'\rangle$.
        Moreover $\cO \cap \bE \langle b_0', \ldots, b_{n}'\rangle$ is the unique extension of $\cO \cap \bE \langle b_0', \ldots, b_{n-1}'\rangle$, so $\tp_L(b_n'/\bE \langle b_0', \ldots, b_{n-1}'\rangle)$ implies the $L_\convex$-type of $b_n'$ over $(\bE \langle b_0', \ldots, b_{n-1}'\rangle, \cO \cap \bE \langle b_0', \ldots, b_{n-1}'\rangle)$.
        It follows that $\tp_{L}(b_n'/\bE\langle b_0', \ldots b_{n-1}'\rangle)$ is not realized in $\bE_2 \langle b_0', \ldots, b_{n-1}'\rangle$ and we are done.
    \end{proof}
\end{lemma}

\begin{corollary}\label{cor:res-dense_ortho}
    Suppose that $\bE \preceq \bU \models T$ and let $\{\cO_i: i \in I\}$ be a family of $T$-convex subrings of $\bU$ such that for all $i \neq j$, $\cO_i \cap \bE \neq \cO_j \cap \bE$. Let also for each $i \in I$, $\bE_i\preceq \bU$ be a $\cO_i$-residually dense extension of $\bE$ with $\cO_i$-res-construction $\overline{x}^{(i)}$. Then, setting $\cO\coloneqq \bigcup_{i \in I} \cO_i$: 
    \begin{enumerate}
        \item $(x^{(i)}_n: i \in I, \, n < |\overline{x}^{(i)}|)$ is a $\cO$-res-construction over $\bE$ and $\dcl(\bigcup_{i \in I} \bE_i)$ is $\cO$-residually cofinal over $\bE$;
        \item $\bigcup_{i \in I} \tp_L(\bE_i/\bE) \vdash \tp_L(\bigcup_{i \in I}\bE_i/\bE)$;
    \end{enumerate}
    \begin{proof}

        We first prove the statement with the additional hypothesis that $I$ is finite.
        If $I$ is empty there is nothing to show, otherwise let $j\in I$ be such that $\cO_j\cap \bE \supsetneq \cO_{i}\cap \bE$ for all $i\in I\setminus \{j\}$.
        By induction we may assume that the statement holds for $I'= I \setminus \{j\}$. So $\bigcup_{i \in I\setminus \{j\}} \tp(\bE_i/\bE) \vdash \tp\big((\bigcup_{i \in I\setminus \{j\}} \bE_i)/ \bE\big)$ and, with $\cO'\coloneqq  \bigcup_{i \neq j} \cO_i$, $\dcl(\bigcup_{i \neq j} \bE_i)$ is $\cO'$-residually cofinal over $\bE$ with $\cO'$-res-construction $(x^{(i)}_n: i \in I\setminus \{j\}, \, n < |\overline{x}^{(i)}|)$.

        But now by Lemma~\ref{lem:res-dense_ortho}~(3) and (2) respectively, it follows that
        \[\tp(\bE_j/\bE) \cup \tp\big((\bigcup_{i \in I\setminus \{j\}} \bE_i)/ \bE\big)\vdash \tp(\bigcup_{i\in I} \bE_i/\bE)\]
        and that $\dcl(\bigcup_{i \in I} \bE_i)$ is $\cO_j$-residually cofinal over $\bE$ with $\cO_j$-res-construction $(x^{(i)}_n: i \in I, \, n < |\overline{x}^{(i)}|)$.

        Now suppose $I$ is infinite. Then (2) follows from a standard compactness argument. As for (1), observe that $(\res_{\cO}(x^{(i)}_n): i \in I, \, n < |\overline{x}^{(i)}|)$ is $\dcl_T$-independent by Lemma~\ref{lem:res-constr_for_different_vals} and by finite character of $\dcl_T$. Similarly $\res_\cO(\dcl(\bigcup_{i \in I}\bE_i))$ must be a cofinal extension of $\res_\cO(\bE)$.
    \end{proof}
\end{corollary}

\begin{definition}\label{def:F-res-dense}
    Let $\bE \preceq \bU \models T$ and $\cF$ be a family of $T$-convex valuation subrings of $\bU$ such that for all $\cO\neq \cO' \in \cF$, $\cO \cap \bE \neq \cO' \cap \bE$. We will call \emph{$\cF$-residually dense} any extension $\bE_*\preceq \bU$ of $\bE$ given as $\bE_*=\dcl \bigcup\{ \bE_\cO: \cO \in \cF\}$ where $\bE_\cO$ is $\cO$-residually dense over $\bE$.
\end{definition}

\begin{remark}\label{rmk:F-res-dense_iso_type}
    In the situation of Definition~\ref{def:F-res-dense}, by Corollary~\ref{cor:res-dense_ortho}, the isomorphism type of $\bE_*$ over $\bE$ only depends on the isomorphism types of the $\bE_\cO$ over $\bE$, which in turn by Lemma~\ref{lem:res_constructible_ext_iso_type} only depends on the isomorphism type of $\res_{\cO}(\bE_\cO)$ over $\res_{\cO}(\bE)$.
\end{remark}

\begin{remark}
    If $\bE_*$ is a $\cF$-residually dense extension of $\bE$ and $\cO$ is any $T$-convex valuation ring of $\bE$, then $\CH_{\bE_*}(\cO)$ is the unique extension of $\cO$ to a $T$-convex valuation ring of $\bE_*$.
    To see this, note first that the extension is unique if and only if $\bE_*$ does not realise the type $\bE < x < \bE^{>\cO}$, so we can reduce to the case where $\bE \preceq \bE_*$ is of finite $\dcl_T$-dimension. Note next that if $\bE \preceq \bE_\dagger \preceq \bE_*$, and $\cO$ has a unique extension to $\bE_\dagger$, say $\cO_\dagger$, and $\cO_\dagger$ has a unique extension to $\bE_*$, then $\cO$ must have a unique extension to $\bE_*$. Thus, we can reduce further to the case where $\bE\preceq \bE_*$ is $1$-dimensional. In particular, we may assume $\bE\preceq \bE_*$ is $\cO'$-residually dense for some $\cO' \in \cF$. If $\cO' \cap \bE \subsetneq \cO$, then Lemma~\ref{lem:res-dense_ortho} implies that the extension is $\cO$-residually cofinal, whence the result follows. If they are equal, the result follows from Fact~\ref{fact:cuts_tricho}. Finally, if $\cO' \cap \bE \supsetneq \cO$, then any $\cO'$-residually dense extension cannot realize the cut $\bE < x < \bE^{>\cO}$ because its realization would be $\cO''$-residual where $\cO''\coloneqq \{t \in \bU: |t|< \bE^{>\cO}\}\subsetneq \cO'$.
\end{remark}

\begin{corollary}\label{cor:type_orthogonal_basis}
    Let $\bE \preceq \bU \models T$ and $\cF$ be a family of $T$-convex valuation subrings of $\bU$. Every $\mathcal{F}$-residually dense extension $\bE \preceq \bE_*\preceq \bU$ has a $\dcl_T$ basis $B$ over $\bE$ such that for all $b \in B$, $\tp(b/\bE)\vdash \tp(b/\bE \langle B \setminus \{b\} \rangle)$.
\end{corollary}

\begin{definition}[Tressl, \cite{tressl2006pseudo}]
    Let $\bE \models T$, $\cF$ be a family of $T$-convex valuation rings of $\bE$, $\bU \succ \bE$ be $|\bE|^+$-saturated, and $\cF'\coloneqq \{\CH_\bU(\cO): \cO \in \cF\}$. Let also for each $\cO\in\cF'$, $\bE_\cO$ be a maximal $\cO$-residually dense extension of $\bE$ in $\bU$.
    The $\cF'$-residually dense extension $\bE_*\coloneqq \dcl(\bigcup_{\cO \in \cF'} \bE_\cO)$ of $\bE$ is called a \emph{pseudo-completion} of $\bE$ with respect to $\cF$.
\end{definition}

\begin{corollary}[Tressl, 4.1 in \cite{tressl2006pseudo}]
    If $\cF$ is a family of $T$-convex valuation rings of $\bE$, then the pseudocompletion of $\bE$ with respect to $\cF$ is unique up to isomorphism over $\bE$.
    \begin{proof}
        Follows from Corollary~\ref{cor:Cauchy-completion_unique} and Remark~\ref{rmk:F-res-dense_iso_type}.
    \end{proof}
\end{corollary}

\begin{proposition}\label{prop:pseudo-completions_subs}
     Let $\bE \prec \bU\models T$ and $\cF\coloneqq \{\cO_i : 0 \le i\le n\}$ be strictly increasing sequence of $T$-convex valuation rings of $\bU$ such that $\bE \cap \cO_i \neq \bE \cap \cO_j$ for all $i\neq j$ and $\cO_n\supseteq \bE$.
     For each $i\le n$, let $\bE_i$ be a $\cO_i$-residually dense extension of $\bE$ and let $\bE_*\coloneqq \dcl(\bigcup_{i \le n}\bE_i)$ be $\cF$-residually dense over $\bE$. Moreover, for each cofinal extension $\bK \preceq \bK'$ of models of $T$, denote by $d(\bK, \bK')$ the $\dcl_T$-dimension of $\bK'$ over the maximal dense extension of $\bK$ within $\bK'$. Then the following are equivalent:
    \begin{enumerate}
        \item if $\bE \preceq \bE_\dagger \preceq \bE_*$ and for all $i\le n$, $\res_{\cO_i}(\bE_\dagger)$ contains $\res_{\cO_i}(\bE_i)$, then $\bE_\dagger$ is isomorphic to $\bE_*$ over $\bE$;
        \item for all $0 \le i < n$, we have $d(\res_{\cO_{i+1}}(\bE_*),\res_{\cO_{i+1}}(\bE))\le \aleph_0$ or $\val(\bE_*, \cO_i)$ does not contain an uncountable well ordered subset.
    \end{enumerate}
    \begin{proof}
        $(2) \Rightarrow (1)$. We induct on $n$.
        Notice that since $\bE_*$ is a cofinal extension of $\bE$, it follows that $\bE_* \subset \cO_n$. Then, the hypothesis $\res_{\cO_n}(\bE_n) \subset \res_{\cO_n}(\bE_\dagger)$ directly implies that $\bE_n \subset \bE_\dagger$.
        So if $n=0$ we are done, and we can assume $n> 0$.

        By repeatedly applying Lemma~\ref{lem:res-dense_ortho}~(2), for all $0\leq i<n$ we know that \break $\dcl(\bE_n,\bE_0,\bE_2, \cdots \bE_i)$ is $\cO_{i}$-residually cofinal over $\bE_n$ so in particular $\bE_*$ is $\cO_{n-1}$-res-constructible over $\bE_n$. Now since $\dim_{\dcl}(\bE_*/\bE_n)\le \aleph_0$ or $\val_{\cO_{n-1}}(\bE_*)$ does not contain uncountable well ordered sets, by Theorem~\ref{thm:mainthm}, it follows that $\bE_\dagger$ is also $\cO_{n-1}$-res-constructible over $\bE_n$.
        
        Since by Lemma~\ref{lem:res_constructible_ext_iso_type} the isomorphism type of $\bE_\dagger$ over $\bE_n$ only depends on the extension $\res_{\cO_{n-1}}(\bE_\dagger)\succeq \res_{\cO_{n-1}}(\bE_n)=\res_{\cO_{n-1}}(\bE)$, it suffices to show that $\bE_\dagger'\coloneqq \res_{\cO_{n-1}}(\bE_\dagger)$ is isomorphic to $\bE_*'\coloneqq \res_{\cO_{n-1}}(\bE_*)$ over $\bE'\coloneqq \res_{\cO_{n-1}}(\bE)$. Now let for all $i<n$, $\cO_{n-1}'\coloneqq \res_{\cO_{n-1}}(\cO_i)$ and observe that for all $i<n$, $\res_{\cO_{i}'}(\bE_*')=\res_{\cO_{i}}(\bE_*)$, $\res_{\cO_{i}'}(\bE_\dagger')=\res_{\cO_{i}}(\bE_\dagger)$, and there is a natural inclusion $\val_{\cO_{i}'}(\bE_*') \subseteq \val_{\cO_{i}}(\bE_*)$.
        Thus we can apply the inductive hypothesis to conclude that $\bE_*'$ is isomorphic to $\bE_\dagger'$ over $\bE'$.

        $(\lnot 2) \Rightarrow (\lnot 1)$. Again we proceed by induction on $n$. Notice that if $n=0$, then (2) always holds, so we can assume $n> 0$. Thus let $i<n$ be the maximum $i$ such that $d(\res_{\cO_{i+1}}(\bE_*), \res_{\cO_{i+1}}(\bE))> \aleph_0$ and $\val(\bE_*, \cO_i)$ contains an uncountable well ordered subset.
        If $i=n-1$, then by Theorem~\ref{thm:mainthm_negative}, we can find $\bE_\dagger$ as in (1) which is not $\cO_{n-1}$-constructible over $\bE_n$. But then, if $\phi : \bE_* \rightarrow \bE_\dagger$ is an isomorphism over $\bE$, then $\phi$ must fix $\bE_n$.
        So if $\ox$ is a res-construction for $\bE_*$ over $\bE_n$, then $\phi(\ox)$ is a res-construction for $\bE_\dagger$ over $\bE_n$, a contradiction.
        If instead $i<n-1$, then $\dim_{\dcl}(\bE_*/\bE_n)\le \aleph_0$ or $\val(\bE_*,\cO_{n-1})$ does not contain an uncountable well ordered set, so by Theorem~\ref{thm:mainthm} any $\bE_\dagger$ as in (1) is $\cO_{n-1}$-res-constructible over $\bE_n$, and we can reduce to the inductive hypothesis. In fact we then have that the isomorphism type of $\bE_\dagger$ over $\bE_n$ is determined by the isomorphism type of $\bE_\dagger'\coloneqq \res_{\cO_{n-1}}(\bE_{\dagger})$ over $\bE'\coloneqq \res_{\cO_{n-1}}(\bE_n)=\res_{\cO_{n-1}}(\bE)$ and that, setting as before $\cO_i'\coloneqq  \res_{\cO_{n-1}}(\cO_i)$:
        \begin{itemize}
            \item For any $0\leq i\leq n-1$, $\res_{\cO_{i}'}(\bE_*')=\res_{\cO_{i}}(\bE_*)$, $\res_{\cO_{i}'}(\bE')=\res_{\cO_{i}}(\bE)$ and the cokernel of the natural inclusion $\val_{\cO_{i}'}(\bE_*') \subseteq \val_{\cO_{i}}(\bE_*)$ is $\val_{\cO_{n-1}}(\bE_*)$ which does not contain uncountable well orders by hypothesis, so $\val_{\cO_{i}'}(\bE_*')$ contains an uncountable well order if and only if $\val_{\cO_{i}}(\bE_*)$ does.
        \end{itemize}
        Thus, by induction there exists some $\bE_\dagger' \prec \bE_*'$, such that for all $0\leq i\leq n-1$, $\res_{\cO_i'}(\bE_i') \subset \res_{\cO_i'}(\bE_\dagger')$, and which is non-isomorphic to $\bE_*'$ over $\bE'$. Let $\ox$ be a tuple in $\bE_*$ such that $\res_{\cO_{n-1}}(\ox)$ is a $\dcl$-basis for $\bE_\dagger'$ over $\bE'$, and let $\bE_\dagger = \bE_n\langle \ox \rangle$. Then, for any $0 \leq i \leq n-1$, $\res_{\cO_i}(\bE_\dagger) = \res_{\cO_i'}(\bE_\dagger') \supset \res_{\cO_i'}(\bE_i') = \res_{\cO_i}(\bE_i)$, and $\bE_n \subset \res_{\cO_n}(\bE_\dagger)$, so $\bE_\dagger$ satisfies the conditions of (1). Finally, $\res_{\cO_{n-1}}(\bE_\dagger) = \bE_\dagger'$ is not isomorphic to $\res_{\cO_{n-1}}(\bE_*)$ over $\res_{\cO_{n-1}}(\bE)$, so in particular $\bE_\dagger$ is not isomorphic to $\bE_*$ over $\bE$.
    \end{proof}
\end{proposition}

\begin{remark}
    Notice that in Proposition~\ref{prop:pseudo-completions_subs}, the hypothesis that $\cO_n\supseteq \bE$ is not essential, as we can always assume that $\bE_n$ is the trivial extension. 
    Also notice that for all $i$, $\res_{\cO_i}(\bE_i)$ is going to be a maximal dense extension of $\res_{\cO_i}(\bE)$ in $\res_{\cO_i}(\bE_*)=\res_{\cO_i}\dcl(\bigcup_{j\le i} \bE_j)$. Thus \[d(\res_{\cO_i}(\bE_*), \res_{\cO_i}(\bE))=\sum_{j<i}\dim_{\dcl_T}(\bE_j/\bE).\]
    In particular in (2) of Proposition~\ref{prop:pseudo-completions_subs}, if there is $i<n$ such that 
    \[\dim_{\dcl_T}(\res_{\cO_{i+1}}(\bE_*)/ \res_{\cO_{i+1}}(\bE))\le \aleph_0,\]
    then for all $j\le i$, it follows that $d(\res_{\cO_{j+1}}(\bE_*), \res_{\cO_{j+1}}(\bE) )\le \aleph_0$.

\end{remark}

It is natural to ask whether the conditions on $\val_{\cO}(\bE_*)$ in point (2) of Proposition~\ref{prop:pseudo-completions_subs}, can be substituted with conditions on $\val_{\cO}(\bE)$, given that $\bE_*$ is a very ``narrow'' extension of $\bE$. We answer this conditionally to some properties of $T$.

\begin{lemma}\label{lem:pbdd_res_dense_ext_vg}
    Let $T$ be power bounded, $\bE\prec \bU \models T$, and $\cO, \cO'$ be $T$-convex valuation subrings of $\bU$. If $\bE_*\subseteq \bU$ is a $\cO'$-residually dense extension of $\bE$, then $\val_\cO(\bE)\cong\val_{\cO}(\bE_*)$.
    \begin{proof}
        When $\cO'\subseteq \cO$, then by Lemma~\ref{lem:res-constr_for_different_vals}, it follows that $\bE_*$ is $\cO$-res-constructible over $\bE$ and the thesis follows from the rv-property (here Fact~\ref{fact:rv-property}).
        Suppose instead $\cO' \supsetneq \cO$.
        Let $\bK \coloneqq  \res_{\cO'}(\bE)$, $\bK_*\coloneqq \res_{\cO'}(\bE_*)$, $V\coloneqq \res_{\cO'}(\cO)$.
        The inclusion $\val_{\cO'}(\bE) \subseteq \val_{\cO'}(\bE_*)$ induces a morphism of short exact sequences
        \[\begin{tikzcd}
            0 \ar[r] &\val_V(\bK)\ar[d] \ar[r] &\val_{\cO}(\bE) \ar[r]\ar[d]& \val_{\cO'}(\bE) \ar[r]\ar[d] &0\\
            0 \ar[r] &\val_V(\bK_*) \ar[r] &\val_{\cO}(\bE_*) \ar[r]& \val_{\cO'}(\bE_*) \ar[r] &0.
        \end{tikzcd}\]
        Now since $\bK_*$ is a dense extension of $\bK\coloneqq \res_{\cO'}(\bE)$, by Lemma~\ref{lem:res-dense_ortho_one}, the leftmost vertical map is an isomorphism. On the other hand by the case $\cO' \subseteq \cO$, proved above, also the rightmost vertical map must be an isomorphism. So, by the short five lemma, the middle vertical map is an isomorphism.   
    \end{proof}
\end{lemma}

\begin{corollary}
    Let $T$ be power bounded, $\bE\prec \bU \models T$, and $\cF$ be a family of $T$-convex valuation rings of $\bU$ such that $\cO\cap \bE \neq \cO' \cap \bE$ for all $\cO \neq \cO'$ in $\cF$. Suppose that $\bE_*$ is a $\cF$-residually dense extension of $\bE$. Then $\val(\bE, \cO)=\val(\bE_*, \cO)$ for all $T$-convex valuation rings $\cO$.
    \begin{proof}
        By definition of $\cF$-residually dense extension there are $\cO$-residually dense extensions $\bE_\cO$ for all $\cO \in \cF$, such that $\bE_*=\dcl(\bigcup_{\cO \in \cF}\bE_\cO)$.
        Notice that by the finite character of $\dcl$ we can reduce to the case in which $\cF=\{\cO_i: i\le n\}$ for some $n\in \omega$ and without loss of generality $\cO_i \supsetneq \cO_j$ for all $0 \le i < j \le n$.
        Then set $\bE_i\coloneqq \dcl(\bigcup_{i<j} \bE_{\cO_j})$ for $0 \le i\le n+1$ and notice that by Corollary~\ref{cor:res-dense_ortho}, each $\bE_{i+1}$ is a $\cO_i$-residually dense extension of $\bE_i$, moreover $\bE_0=\bE$ and $\bE_{n+1}=\bE_*$. At this point it suffices to apply Lemma~\ref{lem:pbdd_res_dense_ext_vg} to conclude that for any $T$-convex valuation ring $\cO'$, $\val_{\cO'}(\bE_i)=\val_{\cO'}(\bE)$.
    \end{proof}
\end{corollary}

\begin{remark}\label{rmk:exponential_ctbl_condition}
    When $T$ is not power-bounded, for a model $(\bE, \cO)\models T_\convex$, we have that $\val_\cO(\bE)$ contains an uncountable well ordered subset if and only if $\bE$ itself contains an uncountable well ordered subset.

    In fact if $\bE$ contains an uncountable well ordered set, then either $\res_\cO(\bE)$ or $\val_\cO(\bE)$ must contain an uncountable well ordered set (argue as in Lemma~\ref{lem:countable_residual_cuts}). On the other hand if $\bE$ defines an exponential and $\res_{\cO}(\bE)$ contains an uncountable well ordered set $S$ and $d \in \bE^{>\cO}$, then $\val_\cO(\exp(Sd))$ is an uncountable well ordered set.

\end{remark}

Thus, in view of Remark~\ref{rmk:exponential_ctbl_condition}, when $T$ is not power-bounded, to see whether the condition that $\val_{\cO_i}(\bE)$ does not contain any uncountable well ordered set implies that $\val_{\cO_i}(\bE_*)$ does not contain any uncountable well ordered set we can just ask the same question about $\bE$ and $\bE_*$. This leads to consider the following property.

\begin{definition}
    We say that a complete o-minimal theory $T$ has \emph{property C} if for every $\bE \models T$, if $\bE$ does not contain any uncountable well ordered set, then so does any $1$-$\dcl_T$-dimensional extension $\bE \langle x \rangle$ of $\bE$.
\end{definition}

\begin{lemma}
    Let $\bE \preceq \bU \models T$, $\cF$ be a family of $T$-convex subrings of $\bU$ such that $\cO \cap \bE \neq \cO' \cap \bE$ for all $\cO \neq \cO'$ in $\cF$, and let $\bE_* \preceq \bU$ be $\cF$-residually dense over $\bE$. Suppose furthermore that $T$ has property C. If $\bE_*$ contains an uncountable well-ordered set, then so does $\bE$.
    \begin{proof}
        Suppose toward contradiction that $\bE_*$ contains an uncountable well-ordered set and $\bE$ does not.

        By Corollary~\ref{cor:type_orthogonal_basis} we can find a $\dcl_T$-basis $B$ of $\bE_*$ over $\bE$ such that for all $b \in B$, $\tp(b/\bE)\vdash \tp(b/\bE \langle B \setminus \{b\} \rangle)$.

        Let $(y_i)_{i<\aleph_1}$ be a sequence in $\bE_*$ such that $y_i$ is strictly increasing. Since $\bE$ does not contain an uncountable well ordered set, there must be $i_0$ such that for all $i>i_0$, $\tp(y_i/\bE)=\tp(y_{i_0}/\bE)$.
        
        For each $i$, write $y_i\coloneqq f_i(\overline{b}^{(i)})$ for some tuple $\overline{b}^{(i)}$ from $B$ and some $\bE$-definable function $f_i$.
        The orthogonality property of $B$ entails that there is a non-empty $B_0\subseteq B$ such that for every $\bE$-definable $f$, if $\tp(f(\overline{b})/\bE)=\tp(y_{i_0}/\bE)$ for some tuple $\overline{b}$ from $B$, then $\overline{b}\supseteq B_0$. Thus all tuples $\overline{b}^{(i)}$ for $i>i_0$ must contain $B_0$.

        Since $B_0$ is finite and thus by property C, $\bE \langle B_0\rangle$ still does not contain an uncountable well-ordered set, we can repeat the argument on the sequence $(y_i)_{i_0\le i<\aleph_1}$, but this time over $\bE \langle B_0\rangle$, instead of $\bE$.
        
        Thus by induction we get an increasing sequence $(i_{n}: n < \omega)$ in $\aleph_1$, such that for $i\ge i_n$, $\overline{b}^{(i)}$ must contain more than $n$ distinct elements from $B$. Since $\{i_{n}: n < \omega\}$ is bounded in $\aleph_1$ we get a contradiction.
    \end{proof}
\end{lemma}

\begin{corollary}
    If $T$ is power bounded or has property C, then in Proposition~\ref{prop:pseudo-completions_subs}(2) we can replace ``$\val(\bE_*, \cO_i)$ does not contain an uncountable well ordered set'' with ``$\val(\bE, \cO_i)$ does not contain an uncountable well ordered set''.
\end{corollary}

We conclude with the following theorem, showing that every countable complete o-minimal theory has property C.

\begin{theorem}
    Let $T$ be a complete \emph{countable} o-minimal theory and suppose that $\bK\models T$ does not contain any uncountable well-ordered set. Suppose that $\bK\langle x \rangle\succ \bK$ is an elementary extension of $\bK$ of $\dcl_T$-dimension $1$. Then $\bK \langle x \rangle$ does not contain any uncountable well-ordered set.
    \begin{proof}
        Assume the contrary for contradiction that there is a strictly increasing $\aleph_1$-indexed sequence $(z_\alpha)_{\alpha<\aleph_1}$ in $\bK \langle x\rangle$.

        Since $T$ is countable, up to extracting a cofinal subsequence of $(z_\alpha)_{\alpha<\aleph_1}$, we can assume that there are $n \in \bN$, a countable $\bP\prec \bK$, an $(n+1)$-ary $\bK_0$-definable function $f$ and a sequence $(\overline{a}^{(\alpha)})_{\alpha<\aleph_1}$ of elements of $\bK^n$ such that $(f(\overline{a}^{(\alpha)}, x): \alpha<\aleph_1)$ is strictly increasing.

        We can also assume that such data was picked so that $n$ is minimum.

        Now let $\bK_0 \prec \bK$ be such that $\tp(x/\bP)\models \tp(x/\bK)$, $\bP \subseteq \bK_0$, and $|\bK_0|<\aleph_0$. Such a $\bK_0$ exists because $\bK$ does not contain any uncountable well order and thus there are countable subsets $A, B\subseteq \bK$ such that $A$ is cofinal in $\bK^{<x}$ and $B$ is coinitial in $\bK^{>x}$, then $\bK_0\coloneqq \bP\langle A\cup B\rangle$ satisfies the requirements. Finally for all $\alpha< \aleph_1$ set $\bK_\alpha \coloneqq \bK \langle \overline{a}^{(\beta)}: \beta < \alpha\rangle$.

        \begin{claim}
            If for some $\alpha <\aleph_1$, $\{a_1^{(\beta)}: \beta < \aleph_1\}$ realizes uncountably many distinct types over $\bK_\alpha$, we reach a contradiction.
            \begin{proof}
                We first build inductively a cofinal subsequence of $(\overline{a}^{(i(\beta))}: \beta <\aleph_1)$ of $(\overline{a}^{(\beta)}: \beta <\aleph_1)$ as follows.

                Suppose that $(i(\beta): \beta<\gamma)$ has been determined. Since $\bK_\alpha\cup \bigcup_{\beta<\gamma}\bK_{i(\beta)}$ is countable, by the claim assumption, there must be some $\delta<\alpha$ such that $a_1^{(\delta)}$ realizes a type over $\bK_\alpha$ that was not realized in $\bK_\alpha \cup \bigcup_{\beta<\gamma}\bK_{i(\beta)}$. Set then $i(\gamma)\coloneq\delta$.

                Now we show that we can inductively construct a sequence $(\overline{b}^{(\beta)}:\beta< \aleph_1)$ s.t.\ $f(\overline{b}^{(\beta)}, x)$ is strictly increasing and for each $\beta<\aleph_1$, $b_1^{(\beta)} \in \bK_\alpha$ and there is some $\delta$ s.t.\ $f(\overline{b}^{(\beta)}, x)<f(\overline{a}^{(\delta)},x)$.

                Assume that $(\overline{b}^{(\beta)}: \beta<\gamma)$ was constructed, then we can find $\delta$ such that for all $\beta<\gamma$, $f(\overline{b}^{(\beta)}, x)<f(\overline{a}^{(i(\delta))},x)$. Set $M\coloneqq \bK_\alpha \langle \overline{a}^{i(\delta)}, \overline{a}^{i(\delta+2n+2)}\rangle$. Clearly $\dim_{\dcl_T}(M/\bK_\alpha)\le 2n$, but by construction $\{\tp(a_1^{(i(\delta+j))}(\var{x}_i)/\bK_\alpha): 1 \le j \le 2n+1\}$ are setwise orthogonal types, so any extension of $\bK_\alpha$ realizing all of them must have dimension at least $2n+1$, whence there is some $j \in \{1, \ldots, 2n +1\}$ such that $\tp(a_1^{i(\delta+j)}/\bK_\alpha)$ is not realized in $M$.

                Now let $\phi(y,r,s)$ be the $M$-formula
                \[\phi(y,r,s) = \exists \overline{z},\, \forall w \in [r,s],\; f(\overline{a}^{(i(\delta))}, w)<f(y, \overline{z},w)<f(\overline{a}^{(i(\delta+2n+2))},w).\]
                There is an interval $[r,s]$ containing $x$, s.t. $\bK \models \phi(a_1^{i(\delta+j)}, r,s)$ and since $\bK_\alpha\supseteq \bK_0$, and $\tp(x/\bK_0)\vdash \tp(x/\bK)$ we can assume that $[r,s]$ is $\bK_\alpha$-definable.

                Now since $\tp(a_1^{i(\delta+j)}/\bK_\alpha)$ is not realized, in $M$ there is an $M$-definable interval $[a,b]$ containing $a_1^{i(\delta+j)}$ such that $\bK\models \forall t \in [a,b],\; \phi(t,r,s)$.

                We pick $b_1^\gamma \in [a,b] \cap \bK_\alpha$ and note that since $\bK\models \phi(b_1^{\beta},r,s)$, there are $b_2^{(\gamma)}, \ldots, b_{n}^{(\gamma)}$ such that for all $w \in [r,s]$,
                \[f(\overline{a}^{(i(\delta))},w)<f(\overline{b}^{(\gamma)})<f(\overline{a}^{(i(\delta+2n+2))},w),\]
                whence the chosen tuple $\overline{b}^{\gamma}$ satisfies the requirements.

                Finally we observe that since each $b_1^{\alpha}$ is in $\bK_\alpha$ and $\bK_\alpha$ is countable, there must be a cofinal subsequence $(\overline{b}^{(l(\beta))}:\beta<\aleph_1)$ of $(\overline{b}^{(\beta)}:\beta<\aleph_1)$ with $b_1^{(l(\beta))}=b_1^{(l(0))}$ for all $\beta<\aleph_1$, thus setting $f_2(\overline{z},w)\coloneqq f(b_1^{l(0)}, \overline{z},w)$ we have that $\bK_\alpha$ is countable, $f_2$ is $n$-ary and $\bK_\alpha$-definable, and $(f_2(b_2^{(l(\beta))}, \ldots, b_n^{(l(\beta))}, x):\beta<\aleph_1)$ is strictly increasing, thus contradicting the minimality of $n$.
            \end{proof}
        \end{claim}

        \newcommand{\TP}{\mathrm{TP}}
        Thus, we can reduce to the case where over every $\bK_\alpha$, the set $\{a_1^{\beta}:\beta<\aleph_1\}$ realises at most $\aleph_0$ many types.
        Let 
        \[\TP_{\aleph_1}(K_\alpha)=\big\{p \in S_1(\bK_\alpha): |\{a_1^\beta: \beta < \aleph_1,\, p(a_1^\beta)\}|>\aleph_0\big\}\]
        be the collection of types over $\bK_\alpha$ that have uncountably many realizations in $\{a_1^\beta: \beta < \aleph_1\}$. 
        Then, consider the function
        \begin{align*}
        g : \aleph_1 &\rightarrow \{0\} \cup \mathbb{N} \cup \{\aleph_0\}\\
        \alpha &\mapsto |\TP_{\aleph_1}(K_\alpha)|.
        \end{align*}
    
        Note that $g$ is increasing, as for any $\alpha<\beta$, every type $p \in TP_{\aleph_1}(\bK_\alpha)$ must extend to at least one type $p'$ over $\bK_{\beta}$, and further it must extend to one in $\TP_{\aleph_1}(K_\beta)$, by Claim~1. We note also that no two distinct types over $\bK_\alpha$ can extend to the same one over $\bK_\beta$. This implies $g$ is eventually constant. Also, we know $g$ is always at least $1$. Now consider two cases (either the eventual value of $g$ is a finite cardinal, or $\aleph_0$) and reach a contradiction for both cases.

        \begin{claim}
        If $g$ is eventually identical to the constant function $m$, for some $m \in \mathbb{N}$, we reach a contradiction.
        \end{claim}
        \begin{proof}
            Suppose toward contradiction that $g$ eventually equals $m \in \bN$. Let $g(\alpha) = m$. Let $p_\alpha \in \TP_{\aleph_1}(K_\alpha)$.
            In this case, by a pigeonhole argument, for each $\alpha<\beta < \aleph_1$, $p_\alpha$ must have a unique extension $p_\beta \in \TP_{\aleph_1}(K_\beta)$. Notice however that since $\bK$ does not contain uncountable well orders, there must be $\beta$ such that for all $\gamma>\beta$, $p_\beta \vdash p_\gamma$. But since no $p_\gamma$ is isolated, we would have that for every $\gamma>\beta$,  $p_\beta(y) \vdash p_\gamma(y) \vdash y \notin \bK_\gamma$ and we would have $p_\beta (\bK)=\emptyset$, yielding a contradiction.
        \end{proof}

        \begin{claim}
            If $g$ is eventually identical to the constant function $\aleph_0$, we reach a contradiction.
        \end{claim}
        
        \begin{proof}
        Suppose so. By enlarging $\bK_0$, we may assume $g$ is always equal to $\aleph_0$. We will inductively construct an increasing sequence of countable ordinals $(C_\alpha)_{\alpha<\aleph_1}$, and a sequence of tuples $(\ob^{(\alpha)})_{\alpha<\aleph_1}$ such that for all $\alpha$,
        \begin{enumerate}
        \item $b_1^{(\alpha)} \in \bigcup_{\beta<\alpha} \bK_{C_\beta}$.
        \item For all $\beta<\alpha$,
        \[f(\oa^{(C_\beta)},x) < f(\ob^{(\alpha)},x) \quad \text{and} \quad f(\ob^{(\alpha)},x)<f(\oa^{(C_\alpha)},x).\]
        \end{enumerate}
        Note that these conditions are sufficient for $f(\ob^{(\alpha)},x)$ to have order type $\aleph_1$.
        
        Suppose that for some $\alpha$, we have built $(\overline{b}^{(\beta)}:\beta<\alpha)$ and $(C_\beta: \beta<\alpha)$ satisfying the requirements.
        
        Let $M \coloneqq \bigcup_{\beta<\alpha} \bK_{C_\beta}$. By the assumption on $g$, $\TP_{\aleph_1}(M)$ has cardinality $\aleph_0$.
        Choose $i_2>i_1\geq \sup_{\beta<\alpha}(C_\beta)$ such that the set $\{a_1^{(\gamma)} : i_1<\gamma<i_2\}$ realises every element of $\TP_{\aleph_1}(M)$. Let the realising indices be $\delta_1,\delta_2, \ldots$. Consider the following first-order formula
        \[
        \phi(y,r,s) = \exists \oz. \forall w \in [r,s]. f(\oa^{(i_1)},w)<f(y,\oz,w)<f(\oa^{(i_2)},w)
        \]
        Note that since $x\notin \bK$, for every $i_1<\gamma<i_2$, there are $r,s \in \bK$ such that $\bK \models \phi(a_1^{\gamma},r,s)$.
        Note also recall that $\tp(x/\bK_0) \vdash \tp(x/\bK)$, so we can assume $r,s \in \bK_0$.
        By o-minimality, there is some constant $N$ such that for any fixed $r,s$, $\phi(y,r,s)$ defines a subset of $\bK$ consisting of at most $N$ points and intervals.
        Choose some $r,s \in \bK_0$ close enough to $x$ so that $\bK \models \phi(a_1^{(\delta_l)},r,s)$ for all $1\leq l \leq N+1$. Thus, since these are $N+1$ distinct points, two of them must be in the same $\bK$-definable interval $I$ such that $\bK \models \forall t \in I,\,\phi(t,r,s)$. But since these two have distinct cuts over $M$, $I$ must contain some point in $M$.
        We will suggestively call this point $b_1^{(\alpha)}$. Then, choose points $b_2^{(\alpha)}, \ldots,  b_n^{(\alpha)}$ in $\bK$, witnessing the existential quantifier in $\phi(b_1^{(\alpha)},r,s)$, i.e.\ such that
        \[\forall w \in [r,s],\; f(\oa^{(i_1)},w)<f(b_1^{(\alpha)}, b_2^{(\alpha)},\ldots, b_{n}^{(\alpha)},w)<f(\oa^{(i_2)},w).
        \]
        Note that since $f(\overline{b}^{(\alpha)},x)>f(\overline{a}^{(i_1)},x)$, for all $\beta<\alpha$, $f(\overline{b}^{(\alpha)},x)>f(\overline{a}^{C_\beta},x)$, as desired. Thus, if we set $C_\alpha=i_2$, $\overline{b}^{(\alpha)}$ and $C_\alpha$ fulfil (2) above. Also observe that each $b_1^{(\alpha)}$ is definable over $\bigcup_{\beta<\alpha} \bK_{C_\beta}$, so we have (1).
        
        Since definable closure is finitary, for each $\alpha>0$, there is some $h(\alpha)<\alpha$ such that $b_1^{(\alpha)}\in\bK_{C_{h(\alpha)}}$. Thus, since the resulting function $h:\aleph_1\to \aleph_1$ is strictly below the identity, by Fodor's Lemma there is some $C_\alpha$ such that $|h^{-1}(C_\alpha)|>\aleph_0$. 
        This implies that by taking a cofinal subsequence of $(\ob^{(\alpha)} : \alpha<\aleph_1)$, we may reduce to the case where $b_1^{(\alpha)}\in \bK_{C_\alpha}$ for all $\alpha$. But then, by extracting a further cofinal subsequence, we may reduce to the case where all the $b_1^{(\alpha)}$ are equal. As in Claim~1, this contradicts the minimality of $n$.
    \end{proof}
    All cases lead to a contradiction, so the theorem is proved.
    \end{proof}
\end{theorem}

\begin{remark}
    A way to restate the result is the following. If $T$ is a countable o-minimal theory, then for each model $\bE \models T$, $S_1(\bE)$ is first-countable if and only if $S_n(\bE)$ is first-countable for all $n$. 
\end{remark}

\section{When is an extension res-constructible?}\label{sec:removing_ambient_res_construction}
Before this point, we have always been working with an `ambient' res-constructible field (or in Section \ref{sec:Tressl_problem} the $\dcl$ of many res-constructible fields), and considering subfields. Now we examine what happens when this structure is removed. The bulk of this section will be consumed examining a particularly pathological extension of RCFs, which `looks' res-constructible locally, but admits no res-construction. Throughout this section, we will always be working with valued fields with no uncountable well-ordered subset in the value group.

\begin{definition}
Given $\bE \prec \bE_*$, and an intermediate $\bE \prec \bE_1 \prec \bE_*$, say that $\bE_1$ is \textit{weakly orthogonal} in $\bE_*$ iff for any finite extension $\bE_1 \prec \bE_2 \prec \bE_*$, we have that $\bE_1 \prec \bE_2$ is res-constructible.
\end{definition}

\begin{proposition}\label{prop:necessary_condition}
Let $(\bE,\cO) \prec (\bE_*,\cO_*)$ be res-constructible. Then, given any countable subset $S \subset \bE_*$, there is a countable extension $S \subset T \subset \bE_*$ such that $\bE\langle T \rangle$ is weakly orthogonal.
\end{proposition} 
\begin{proof}
Fix a res-construction $\overline{s}$ of $\bE_*$ over $\bE$. Let $T$ be a countable subset of $\overline{s}$ such that $S \subset \bE\langle T \rangle$. Then, since $\bE_*$ is res-constructible over $\bE\langle T \rangle$, by Corollary \ref{cor:finite_dim_fact}, $\bE\langle T \rangle$ is weakly orthogonal, as desired.
\end{proof}

The converse to Proposition \ref{prop:necessary_condition} does not hold, as we will see in Example \ref{ex:Pathological} below. Before presenting the counterexample, we now show that a ``uniform'' version of weak-orthogonality is equivalent to res-constructibility.

\begin{theorem}\label{thm:uniform_condition}
Let $\bE \prec \bE_*$ be an extension such that $\val(\bE_*)$ contains no uncountable ordinal. The following are equivalent:
\begin{itemize}
    \item  There exists a $\dcl$-basis $B$ for $\bE_*$ over $\bE$, and a function
\[
\phi : \bE_* \rightarrow \{ S \subset B : S\text{ countable} \}
\]
such that
\begin{enumerate}
\item For any countable $S \subset \bE_*$, $\bE\langle \phi(x) : x \in S \rangle$.
\item For any $x \in \bE_*$, for any $y \in \phi(x)$,
$\phi(y) \subset \phi(x)$, and $x \in \bE \langle \phi(x) \rangle$.
\end{enumerate}

\item $\bE \prec \bE_*$ is res-constructible.
\end{itemize}
\end{theorem}
\begin{proof}
First, let us prove $\Leftarrow$. Fix a res-construction $\overline{s}$ of $\bE_*$ over $\bE$. Then, let $\phi(x) = \text{the minimal } T \subset \overline{s} \text{ s.t. } x \in \bE\langle T \rangle$. Then, again using Corollary \ref{cor:finite_dim_fact}, for any countable $S \subset \bE_*$, $\bE\langle \phi(S)\rangle$ is weakly orthogonal.
Also, trivially for any $x$ and any $y \in \phi(x)$, $\phi(y) \subset \phi(x)$, and $x \in \bE\langle \phi(x)\rangle$.

Now we will prove $\Rightarrow$. Note that we can extend $\phi$ to the power set of $\bE_*$ by letting $\phi(T) = \bigcup_{x\in T} \phi(x)$. Let $b_\beta : \beta < \alpha$ be an ordinal enumeration of $B$. It suffices to prove that the chain
\[
\bK_\beta = \bE \langle \phi(\{ b_\gamma : \gamma < \beta\}) \rangle : \beta < \alpha
\]
refines to a res-construction of $\bE_*$ over $\bE$.
Note first that each $\bK_\beta \leq \bK_{\beta+1}$ is of countable $\dcl$-dimension by the definition of $\phi$.
Next, we claim $\bK_\beta$ is weakly orthogonal, for all $\beta$.
To see this, suppose there exist some $b_1,\cdots b_n \in \bE_*$, such that $b_n$ is immediate over $\bK_\beta \langle b_1,\cdots b_{n-1} \rangle$.
By Lemma \ref{lem:countable_residual_cuts}, there is some countable $S \subset \bK_\beta$ such that $\tp(b_{n}/\bE\langle S,b_1,\cdots b_{n-1}\rangle) \models \tp(b_{n}/\bK_\beta\langle b_1,\cdots b_{n-1}\rangle)$. 
Therefore, $\tp(b_N/\bK')$ is immediate for any model $\bK'$ such that $\bE\langle S,b_1,\cdots b_{n-1}\rangle \prec \bK' \prec \bK_\beta\langle b_1,\cdots b_{n-1}\rangle$.
In particular, $b_{n}$ is immediate over $\bE \langle \phi(S), b_1, \cdots b_{n-1} \rangle$, which contradicts the hypothesis that $\bE \langle \phi(S) \rangle$ is weakly orthogonal.
Thus, $\bK_\beta \prec \bK_{\beta+1}$ is res-constructible for each $\beta$, so $\bE \prec \bE_*$ is res-constructible, as desired.
\end{proof}

Now, we will introduce an example demonstrating that this uniformity is necessary.

\begin{example}\label{ex:Pathological}
   Let $T\coloneqq \RCF$ be the theory of real closed fields, so that $T_\convex = \RCVF$. Let $\bK \models T$ with $\dim_{\dcl_T}(\bK)=\aleph_2$ and let $\bU\coloneqq \bK\Lhp \mu^{\bQ}\Rhp_{\mathrm{Puiseux}}\succ \bK$ be the field of Puiseux series with coefficients in $\bK$ in some infinitesimal $0<\mu<\bK^{>0}$. Set also $\cO\coloneqq \CH_\bU(\bK)$.
   
   Let $\bP\preceq \bK$ be the real closure of $\bQ$, $\bE\coloneqq \bP \langle \mu\rangle$, and $\bE_1$ be a $\cO$-res-constructible extension of $\bE$ with res-construction $B\coloneqq (b_{i, \alpha}: \alpha< \aleph_2, \, i <\omega) \in \bK^{\omega \times \aleph_2}$.
   
   Also, for each $\alpha<\aleph_2$ denote by $B_\alpha$ the $\omega$-tuple $(b_{i, \alpha}: i\in \omega)$.
   
   Consider the function $M: (\cO\cap \bE_1)^\omega \times (\cO\cap \bE_1)^\omega \to \bU$ defined by \[M((c_i)_{i<\omega}, (d_i)_{i<\omega}) \coloneqq  \sum_{i<\omega} c_i d_i \mu^i.\]
   
   Finally set $\bE_2\coloneqq \bE_1\langle M(B_\alpha, B_\beta): \alpha < \beta <\aleph_2\rangle$ and observe that $\bE_2$ is a dense extension of $\bE_1$ because each $\bE_1 \langle M(B_\alpha, B_\beta)\rangle$ is a dense extension and $\bE_2\subseteq \bU$, so it does not contain any $x>\bE_1$.
\end{example}

\begin{proposition}\label{prop:Example_indpendence+}
   In the above example, the family $\big(M(B_\alpha, B_\beta): \alpha<\beta<\aleph_2\big)$ is $\dcl_T$-independent over 
   \[\bE_1'\coloneqq \bigcup_{\substack{F \subseteq B\\|F|<\aleph_0}} \big(\bP \langle F\rangle \big) \Lhp \mu^\bQ \Rhp.\]
   where $ (\bP\langle F\rangle)\Lhp \mu^\bQ\Rhp$ is the Hahn field with coefficients from $\bP\langle F \rangle \subseteq \bR$ and exponents from $\bQ$. In particular $\bE_1' \cap \bE_2=\bE_1$.
\begin{proof}
    Note $\bE_1 \prec \bE'_1$. By finite character of $\dcl_T$ it suffices to show that for each $n$, $\big(M(B_\alpha, B_\beta): \alpha < \beta<n\big)$ is $\dcl_T$-independent over $\bigcup_{m<\omega} \bP \langle b_{i,j}: i<m,\, j<n\rangle \Lhp \mu^\bQ \Rhp$.

    Let $\overline{\var{x}}\coloneqq (\var{x}_{i,j})_{i<\omega, j<n}$, $\overline{\var{y}}\coloneqq (\var{y}_{i,j})_{i<j<n}$, $\overline{\var{z}}\coloneqq (\var{z}_{i})_{i<l}$ be tuples of formal variables and suppose that $p(\overline{\var{x}}, \overline{\var{y}}, \overline{\var{z}}) \in \bE [\overline{\var{x}}, \overline{\var{y}}, \overline{z}]$
    is such that 
    \[p\big( (b_{i,j})_{i<\omega, j<n}, (M(B_i, B_j))_{i<j<n}, (c_{i})_{i<l}\big)=0\in \bK\Lhp \mu \Rhp.\]
    where $(c_i)_{i<l}$ lives in $\bP \langle b_{i,j}: i<m,\, j<m\rangle \Lhp \mu^\bQ \Rhp$ for some $m$, and is $\dcl_T$-independent over $\bP\langle b_{i,j} : i<\omega, j<n\rangle$. Notice that by definition we can write
    \[c_k = g_k((b_{i,j})_{i<m, j<n})= \sum_{q \in \bQ} g_{k,e} ((b_{i,j})_{i<m, j<n}) \mu^q\]
    for some $g_{k,e}\in \bP[(\var{x}_{i,j})_{i<m, j<n}]\subseteq \bP[\overline{\var{x}}]$.
    
    Let $\tilde{p}(\overline{\var{x}}) \in (\bP [\overline{\var{x}}])\Lhp \mu^\bQ\Rhp$ and $a_q(\overline{\var{x}}) \in \bP[\overline{\var{x}}]$ for $q \in \bQ$ be defined by
    \[\tilde{p}( \overline{\var{x}})\coloneqq 
    p\left(\overline{\var{x}},
        \big(\sum_{k<\omega} \var{x}_{k,i}\var{x}_{k,j} \mu^k\big)_{i<j<n},
        \big(g_k(\overline{x})\big)_{k<l}\right) =
    \sum_{q \in \bQ} a_q(\overline{\var{x}}) \mu^q, \quad a_q(\overline{\var{x}}) \in \bP[\overline{\var{x}}]\]
    so that $0=p\big( (b_{i,j})_{i<\omega, j<n}, (M(B_i, B_j))_{i<j<n}, (c_i)_{i<l}\big)=\tilde{p}( (b_{i,j})_{i<\omega, j<n})$, whence, by algebraic independence of the $b_{i,j}$ over $\bE$, $a_q(\overline{\var{x}})=0$ for all $q \in \bQ$.

    Now suppose toward contradiction that $p \neq 0$. Then by algebraic independence of $(b_{i,j})_{i<\omega, j<n}$ and $(c_k)_{k<l}$, there must be some $0<\alpha<\beta <n$, such that if we define $\overline{\var{y}}'\coloneqq (\var{y}_{i,j})_{(i,j)\neq (\alpha, \beta)}$, there exists $E>0$ and polynomials $q\in \bE[ \ox,\oy,\oz],q_l\in \bE[ \ox,\oy',\oz]$ for $0 \leq l < E$, obeying
    \[p(\overline{\var{x}}, \overline{\var{y}}, \overline{\var{z}})=q(\overline{\var{x}}, \overline{\var{y}}', \overline{\var{z}})\var{y}_{\alpha, \beta}^E + \sum_{0\leq l<E}q_l(\overline{\var{x}}, \overline{\var{y}}', \overline{\var{z}})y_{\alpha,\beta}^{l}\]
    Similarly to $p$ we write
    \[\tilde{q}(\overline{\var{x}})\coloneqq 
    q\left(\overline{\var{x}},
        \big(\sum_{k<\omega} \var{x}_{k,i}\var{x}_{k,j} \mu^k\big)_{(i,j)\neq (\alpha, \beta)},
        \big(g_k(\overline{x})\big)_{k<l}\right)= \sum_{q \in \bQ} d_q(\overline{\var{x}}) \mu^{q}, \quad d_q(\overline{\var{x}}) \in \bP[\overline{\var{x}}]\]
    and let $e\coloneqq \min\{q \in \bQ: d_q (\overline{\var{x}}) \neq 0\}$.
    
    Now let $m$ be an upper bound for the index $i$ of the variables $(\var{x}_{i,j})_{i<\omega, j<n}$ appearing in $d_e(\overline{\var{x}})$ and choose $\overline{r}\coloneqq (r_{i,j})_{i<\omega, j<n}$ in $\bP$ so that
    \[d_e(\overline{r})\neq 0 \qquad \text{and} \qquad \text{for $i>m$,}\; r_{i,j}\coloneqq \begin{cases}
            0 &\text{if}\; j \notin \{\alpha, \beta\}\\
            s_{i} &\text{if}\; j \in \{\alpha, \beta\}.
        \end{cases}\]
    where $(s_i)_{m<i<\omega}$ is a sequence in $\bP$ such that $\sum_{i<\omega} s_i^2\mu^i \notin \bE \langle g_i(\overline{r}) : i <l\rangle$. 
    Notice that this is possible because each $g_k(\overline{r})$ only depends on $(r_{i,j})_{i<m, j<n}$.
    
    With such a choice $\tilde{q}(\overline{r})\in \mu^e d_e(\overline{r}) + \mu^{e} \co$, so $\tilde{q}(\overline{r}) \neq 0$, and moreover setting $\overline{r}_{i}\coloneqq (r_{k,i})_{k< \omega}$, we have:
    \begin{itemize}
        \item $M\big(\overline{r}_i, \overline{r}_j) = \sum_{k\le m} r_{k,i}r_{k,j}\mu^k \in \bE$ for all $(i,j) \neq (\alpha, \beta)$;
        \item $M(\overline{r}_{\alpha}, \overline{r}_\beta)\coloneqq \sum_{k\le m} r_{k,\alpha}r_{k,\beta}\mu^k + \sum_{k>m} s_k^2\mu^k \notin \bE \langle g_i(\overline{r}) : i <l\rangle$.
    \end{itemize}
    But this yields a contradiction because
    \[\begin{aligned}
        0=\tilde{p}(\overline{r}) = \tilde{q}(\overline{r})M(\overline{r}_\alpha, \overline{r}_\beta)^E + \\
        + \sum_{h<E} q_h\big(
            \overline{r},
            \big(M(\overline{r}_i, \overline{r}_j), \big)_{(i,j)\neq (\alpha,\beta)},
            (c_i)_{i<l}
        \big) M(\overline{r}_\alpha, \overline{r}_\beta)^h,
    \end{aligned}\]
    so we would have $M(B_\alpha, B_\beta) \in \bE \langle g_{i}(\overline{r}): i<l\rangle$, because $\bE \langle g_{i}(\overline{r}): i<l\rangle$ is by definition algebraically closed in $\bK \Lhp \mu^\bQ\Rhp$. Thus, $p=0$, so we have the desired $\dcl_T$-independence.

    For the last assertion suppose that $x \in \bE_2 \cap \bE_1'$. Since $x \in \bE_2$ it is algebraic over $\bE_1 (F)$ for some finite $F\subseteq \{M(B_\alpha, B_\beta): \alpha<\beta<\aleph_2\}$. If $x \notin \bE_1$, then by the exchange property, some $y \in F$ would be algebraic over $\{x\} \cup F \setminus \{y\}$, which, since $x \in \bE_1'$, would contradict the algebraic independence of $F$ over $\bE_1'$. Therefore we must have $x \in \bE_1$.
    \end{proof}
\end{proposition}

\begin{corollary}
    In Example~\ref{ex:Pathological}, $\bE_2$ is not $\cO$-res-constructible over $\bE$.
    \begin{proof}
        Suppose toward contradiction that $\bE_2$ is $\cO$-res-constructible over $\bE$. Then by Theorem~\ref{thm:uniform_condition} we can find a $\dcl_T$-basis $C$ of $\bE_2$ over $\bE$, and a function
        \[\phi : \bE_2 \rightarrow C\]
        satisfying (1) and (2) of Theorem~\ref{thm:uniform_condition}.
        
        Let $K\coloneqq \bE \langle \phi(\bigcup_{\alpha <\aleph_1} B_\alpha)\rangle$ so that $\dim_{\dcl_T}(K/\bE)=\aleph_1$, whence it must be contained in some $\bE \langle B_\alpha, M(B_\alpha, B_\beta) : \alpha <\beta<\gamma\rangle$ for some limit ordinal $\gamma$ with $\aleph_1 \leq \gamma<\aleph_2$.
        
        Let $K'\coloneqq  \bE \langle \phi( B_\gamma \cup \bigcup_{\alpha< \aleph_1} B_\alpha) \rangle$.
        Notice that by construction $K'$ is generated over $K$ by $\phi(B_\gamma)$ which is countable.
        
        On the other hand for all $\alpha < \aleph_1$ we have $M(B_\alpha, B_\gamma) \in K'$, because by property (1) of $\phi$, there are no elements of $\bE_2 \setminus K'$ that are weakly immediate over $K'$ and $M(B_\alpha, B_\gamma)$ is weakly immediate over $K'$.
        This yields a contradiction, because by Proposition~\ref{prop:Example_indpendence+}, $\{M(B_\alpha, B_\gamma): \alpha <\gamma\} \subseteq K'$ is $\dcl_T$-independent over $K$, and uncountable.
    \end{proof}
\end{corollary}

\begin{lemma}\label{lem:Example_auto}
    Fix some $b_{l, \delta}$ in Example~\ref{ex:Pathological}. For all $z \in \co \cap \bE_2$ there is a unique automorphism $\phi_z$ of $\bE_2$ over $\bE$ such that
    \begin{enumerate}
        \item $\phi_z(b_{l, \delta}) = b_{l, \delta}+ z$
        \item for all $(i,\alpha)\neq (l, \delta)$, $\phi_z(b_{i,\alpha})=b_{i,\alpha}$.
    \end{enumerate}
    For such $\phi_z$, for all $\alpha,\beta<\aleph_2$, $\phi_z\big(M(B_\alpha, B_\beta)\big)=M(\phi_z(B_\alpha), \phi_z(B_\beta))$. Moreover for all $x \in \bE_2$, the function $\co \cap \bE_2 \ni z \mapsto \phi_z(x)$ extends to a $\bE_2$-definable function on an interval containing $\co$.
    \begin{proof}
        Notice if such $\phi_z$ exists, since $\bE_1$ is dense in $\bE_2$, it is uniquely determined by conditions (1) and (2).
        Set $c_{\alpha,\beta}\coloneqq M(B_\alpha, B_\beta)$ if $\delta \notin \{\alpha,\beta\}$, and $c_{\alpha,\beta}\coloneqq M(B_\alpha, B_\beta)-b_{l, \alpha}b_{l,\beta}+\phi_z(b_{l,\alpha})\phi_z(b_{l,\beta})$ otherwise.
        Notice that in either case as a series in $\mu$, $c_{\alpha,\beta}$ has coefficients in $\bP\langle b_{i,\alpha}: (i,\alpha) \neq (l, \delta)\rangle$, so $\tp\big(c_{\alpha,\beta}/\bE \langle b_{k, \gamma}: (k,\gamma) \neq (l, \delta)\rangle\big)\vdash \tp(c_{\alpha,\beta}/\bE_1)$ has a unique realization in $\bE_2$. Since by definition $\phi_z$ fixes $\bE \langle b_{i,\alpha}: (i,\alpha) \neq (l,\delta)\rangle$, it must thus fix the whole
        \[\bE_2^{-}\coloneqq \bE\langle b_{k,\gamma}, c_{\alpha, \beta}: (k, \gamma) \neq (l, \delta), \, \alpha<\beta<\aleph_2\rangle.\]
        
        To conclude observe that $\bE_2^- \langle b_{l, \delta}\rangle = \bE_2$ and that by construction $\res_\cO(b_{l, \delta}) \notin \res_\cO(\bE_2)$, so $\tp(b_{l, \delta}/\bE_2^-)=\tp(b_{l, \delta}+z/\bE_2^-)$ and we can conclude that $\phi_z$ exists and has the additional property of the statement.

        As for the final clause, fix $x \in \bE_2$ and observe that if $x=g(b_{l, \delta})$ for a $\bE_2^-$-definable function $g$, then $\phi_z(x)=g(b_{l, \delta}+z)$ for all $z \in \bE_2\cap \co$.
    \end{proof}
\end{lemma}

In the setting of Example~\ref{ex:Pathological}, for each $z\in \bE_2$, let us denote by $T_z$ the smallest subset of $\{b_{i,\alpha}: i<\omega, \, \alpha \in \aleph_2\}$ such that $\bP\langle T_z \rangle$ contains all the coefficients of $z$ as a series in $\mu$.
 
\begin{lemma}\label{lem:Example_simplifying_auto}
    In the setting of Example ~\ref{ex:Pathological}, let $z \in \bE_2$ be such that $T_z$ is infinite and suppose $S\subseteq \{b_{i,\alpha}: i<\omega, \, \alpha \in \aleph_2\}$ is such that $T_z \setminus S \neq \emptyset$. Then there is an automorphism $\phi$ of $\bE_2$ fixing $S$ such that $T_{\phi(z)}$ is finite.
    \begin{proof}
        Let $z\coloneqq  \sum_{n}\mu^{a+nb}c_n$, and observe that $T_z=\bigcup T_{c_n}$, so for $n$ large enough $T_{c_n}\setminus S\neq \emptyset$. Fix the minimum such $n$ and let $(l, \delta)$ be such that $b_{l,\delta} \in T_{c_n}\setminus S$, notice in particular that this also entails that $b_{l, \delta} \notin T_{c_k}\subseteq S$ for all $k<n$.
        
        Write $c_{n}\coloneqq f_n(b_{l, \delta})$ for some monotone bijection $f_n$ definable over $\bP\langle b_{i, \alpha}: (i, \alpha) \neq (l, \delta)\rangle$. Since the cut of $b_{l, \delta}$ over $\bP\langle b_{i, \alpha}: (i, \alpha) \neq (l, \delta)\rangle$ is not a principal cut, $\tau(\var{x})\coloneqq  f_n^{-1}(\var{x}+c_n)-b_{l, \delta}$ is a $\bE_2$-definable function such that $\tau(\co)\subseteq \co$.
        
        By the previous Lemma, for every $x\in \bE_2 \cap \co$, there is $\phi_x \in \Aut(\bE_2/S)$ such that $\phi_x(b_{l, \delta})=b_{l, \delta}+\tau(x)$. Since $\phi_x$ fixes $S$ and by the minimality of $n$ it fixes also all $c_k$ for $k<n$, we have
        \begin{align*}
            &\phi_x(z) \in \left(\sum_{k<n}c_k\mu^{a+bk}\right) + \mu^{a+nb}(f_n(\phi(b_{l,\delta}))) +\mu^{a+(n+1)b}\phi(c_{n+1}) + \mu^{a+(n+1)b} \co\\
            &\phi_x(z) \in \left(\sum_{k<n}c_k\mu^{a+bk}\right) + \mu^{a+nb}(c_n+x) \ \ \ \ \ \ +\mu^{a+(n+1)b}c_{n+1} \ \ \ +\mu^{a+(n+1)b} \co
        \end{align*}
        Since $x\mapsto \phi_x(z)$ is definable continuous on $\co$ and $\phi_0(z)=0$ and 
        \[\phi_{-\mu^b(c_{n+1}+1)}(z)\in \left(\sum_{k<n}c_k\mu^{a+bk}\right) +\mu^{a+nb}(c_n) + \mu^{a+(n+1)b} (-1+\co)\]
        there is $x$ such that $\phi_x(z)\coloneqq \sum_{k\le n} c_n \mu^{a+nb}$. Let $\phi\coloneqq \phi_x$ for such $x$.
        We see that it fixes $S$ and $T_{\phi(z)}$ is finite.
    \end{proof}
\end{lemma}

\begin{lemma}\label{lem:finite_ok}
    In the setting of Example~\ref{ex:Pathological} the following hold:
    \begin{enumerate}
        \item if $x \in \bE_2 \setminus \bE_1$, then $T_x$ is infinite;
        \item if $F\subseteq \{b_{\alpha, i}: \alpha <\aleph_2,\, i<\omega\}$ is finite, then $\bE_2 \cap \big(\bP \langle F \rangle\big) \Lhp \mu^\bQ \Rhp= \bE \langle F\rangle$.
    \end{enumerate}
    \begin{proof}
        Notice that if $T_x$ is finite, then $x$ is an element of $\bE_1'$ as defined in Proposition~\ref{prop:Example_indpendence+}.
        On the other hand by the same Proposition $\bE_1' \cap \bE_2 = \bE_1$ and (1) follows.

        As for (2) notice that $\big(\bP \langle F \rangle\big) \Lhp \mu^\bQ \Rhp \subseteq \bE_1'$, so 
        \[\bE_2 \cap \big(\bP \langle F \rangle\big) \Lhp \mu^\bQ \Rhp \subseteq \bE_1 = \bE \langle b_{\alpha, i}: \alpha < \aleph_2,\, i<\omega\rangle\]
        Note also that $(\bP\langle F \rangle)\Lhp \mu^\bQ \Rhp$ is a dense extension of $\bE\langle F \rangle$, whereas $\bE_1$ is a res-constructible extension, so by Lemma ~\ref{lem:res_implies_no_weak_elements} their intersection is $\bE\langle F \rangle$. This concludes the proof.
    \end{proof}
\end{lemma}

\begin{proposition}
    In the setting of Example~\ref{ex:Pathological}, any countable set $S \subseteq \bE_2$ extends to some countable $S'\supseteq S$, such that $\bE \langle S'\rangle$ is weakly orthogonal in $\bE_2$.
    \begin{proof}
        Notice that up to reindexing the blocks, we can assume that $S\subseteq K\coloneqq \bE \langle B_i, M(B_j, B_k): i<\omega, \, j<k<\omega\rangle$. So it suffices to show that $K$ is weakly orthogonal in $\bE_2$.

        Let $F\subseteq\bE_2$ be finite and without loss of generality $\dcl_T$-independent over $\bE$.
        
        We claim that there is $\chi \in \Aut(\bE_2/K)$, such that for all $z \in F$, $T_{\chi(z)}$ is finite. Let $F_1\coloneqq \{z \in F: |T_{z}|<\aleph_0\}$ and $F_0\coloneqq F\setminus F_1$. To prove the claim by induction it suffices to show that for every $z \in F_0$, there is $\phi \in \Aut(\bE_2/K\langle T_{z}: z \in F_1\rangle)$ such that $|T_{\phi(z)}|<\aleph_0$. For this it suffices to apply Lemma~\ref{lem:Example_simplifying_auto}, with $S=\{B_i, M(B_j, B_k): i<\omega, \, j<k<\omega\} \cup \{T_z: z \in F_1\}$.

        Notice that by Lemma~\ref{lem:finite_ok} when $T_z$ is finite, $\bE\langle z\rangle \subseteq \bE \langle T_z\rangle$, whence $K\langle z \rangle \subseteq K \langle T_z\rangle$.
        
        Thus since $\bigcup_{z \in F} T_{\chi(x)}$ is finite we have $K \langle \chi(F) \rangle \subseteq K\langle \bigcup_{z \in F} T_{\chi(z)}\rangle$.
        
        Since the latter is a $\cO$-res-constructible extension of $K$, by Corollary~\ref{cor:finite_dim_fact} we can conclude that also $K\langle \chi(F) \rangle$ and thus $K\langle F \rangle$ are $\cO$-res-constructible over $K$.
    \end{proof}
\end{proposition}
Example \ref{ex:Pathological} shows that we can have a non-res-constructible extension $\bE \prec \bE_2$ which satisfies the conclusion of Proposition \ref{prop:necessary_condition}. This suggests that res-constructibility is not a `local' property.

\bibliographystyle{abbrvnat}
\bibliography{biblio}

@article{dries1995t,
 AUTHOR = {Lou {\noopsort{dries}van den Dries} and Lewenberg, Adam H.},
     TITLE = {{$T$}-convexity and tame extensions},
   JOURNAL = {J. Symbolic Logic},
  FJOURNAL = {The Journal of Symbolic Logic},
    VOLUME = {60},
      YEAR = {1995},
    NUMBER = {1},
     PAGES = {74--102},
      ISSN = {0022-4812,1943-5886},
   MRCLASS = {03C60 (03C10 03C35 12L12)},
  MRNUMBER = {1324502},
MRREVIEWER = {O.\ V.\ Belegradek},
       DOI = {10.2307/2275510},
       URL = {https://doi.org/10.2307/2275510},
}

@article{dries1997t,
AUTHOR = {Lou {\noopsort{dries}van den Dries}},
     TITLE = {{$T$}-convexity and tame extensions. {II}},
   JOURNAL = {J. Symbolic Logic},
  FJOURNAL = {The Journal of Symbolic Logic},
    VOLUME = {62},
      YEAR = {1997},
    NUMBER = {1},
     PAGES = {14--34},
      ISSN = {0022-4812,1943-5886},
   MRCLASS = {03C60 (03C10 03C35 12L12)},
  MRNUMBER = {1450511},
MRREVIEWER = {O.\ V.\ Belegradek},
       DOI = {10.2307/2275729},
       URL = {https://doi.org/10.2307/2275729},
}

@article{marker1994definable,
AUTHOR = {Marker, David and Steinhorn, Charles I.},
     TITLE = {Definable types in {O}-minimal theories},
   JOURNAL = {J. Symbolic Logic},
  FJOURNAL = {The Journal of Symbolic Logic},
    VOLUME = {59},
      YEAR = {1994},
    NUMBER = {1},
     PAGES = {185--198},
      ISSN = {0022-4812,1943-5886},
   MRCLASS = {03C45},
  MRNUMBER = {1264974},
MRREVIEWER = {Alexandre\ Ivanov},
       DOI = {10.2307/2275260},
       URL = {https://doi.org/10.2307/2275260},
}

@article {tressl2006pseudo,
    AUTHOR = {Tressl, Marcus},
     TITLE = {Pseudo completions and completions in stages of o-minimal
              structures},
   JOURNAL = {Arch. Math. Logic},
  FJOURNAL = {Archive for Mathematical Logic},
    VOLUME = {45},
      YEAR = {2006},
    NUMBER = {8},
     PAGES = {983--1009},
      ISSN = {0933-5846,1432-0665},
   MRCLASS = {03C64 (12J10 12J15)},
  MRNUMBER = {2271334},
MRREVIEWER = {Chris\ Miller},
       DOI = {10.1007/s00153-006-0022-2},
       URL = {https://doi.org/10.1007/s00153-006-0022-2},
}

@article{dries2000field,
 AUTHOR = {Lou {\noopsort{dries}van den Dries} and Patrick Speissegger},
     TITLE = {The field of reals with multisummable series and the
              exponential function},
   JOURNAL = {Proc. London Math. Soc. (3)},
  FJOURNAL = {Proceedings of the London Mathematical Society. Third Series},
    VOLUME = {81},
      YEAR = {2000},
    NUMBER = {3},
     PAGES = {513--565},
      ISSN = {0024-6115,1460-244X},
   MRCLASS = {03C64 (03C10 12L12 26E05)},
  MRNUMBER = {1781147},
MRREVIEWER = {Chris\ Miller},
       DOI = {10.1112/S0024611500012648},
       URL = {https://doi.org/10.1112/S0024611500012648},
}

@book{tyne2003t,
 AUTHOR = {Tyne, James Michael},
     TITLE = {T-levels and {T}-convexity},
      NOTE = {Thesis (Ph.D.)--University of Illinois at Urbana-Champaign},
 PUBLISHER = {ProQuest LLC, Ann Arbor, MI},
      YEAR = {2003},
     PAGES = {106},
      ISBN = {978-0496-34011-8},
   MRCLASS = {99-05},
  MRNUMBER = {2704495},
       URL =
              {https://www.proquest.com/docview/305332004},
}

@misc{freni2024t,
      title={T-convexity, Weakly Immediate Types and {$T$}-{$\lambda$}-Spherical Completions of o-minimal Structures}, 
      author={Freni, Pietro},
      year={2024},
      eprint={2404.07646},
      archivePrefix={arXiv},
      primaryClass={math.LO},
      url={https://arxiv.org/abs/2404.07646}, 
}

@article{dries1994real,
 AUTHOR = {Lou {\noopsort{dries}van den Dries} and Chris Miller},
     TITLE = {On the real exponential field with restricted analytic
              functions},
   JOURNAL = {Israel J. Math.},
  FJOURNAL = {Israel Journal of Mathematics},
    VOLUME = {85},
      YEAR = {1994},
    NUMBER = {1-3},
     PAGES = {19--56},
      ISSN = {0021-2172,1565-8511},
   MRCLASS = {03C10 (03C62 12J15)},
  MRNUMBER = {1264338},
MRREVIEWER = {A.\ J.\ Wilkie},
       DOI = {10.1007/BF02758635},
       URL = {https://doi.org/10.1007/BF02758635},
}

@article{wilkie1996model,
 AUTHOR = {Wilkie, A. J.},
     TITLE = {Model completeness results for expansions of the ordered field
              of real numbers by restricted {P}faffian functions and the
              exponential function},
   JOURNAL = {J. Amer. Math. Soc.},
  FJOURNAL = {Journal of the American Mathematical Society},
    VOLUME = {9},
      YEAR = {1996},
    NUMBER = {4},
     PAGES = {1051--1094},
      ISSN = {0894-0347,1088-6834},
   MRCLASS = {03C62 (03C60 03C65 14P15)},
  MRNUMBER = {1398816},
MRREVIEWER = {Luc\ B\'{e}lair},
       DOI = {10.1090/S0894-0347-96-00216-0},
       URL = {https://doi.org/10.1090/S0894-0347-96-00216-0},
}

@incollection {miller1993growth,
    AUTHOR = {Miller, Chris},
     TITLE = {A growth dichotomy for o-minimal expansions of ordered fields},
 BOOKTITLE = {Logic: from foundations to applications ({S}taffordshire,
              1993)},
    SERIES = {Oxford Sci. Publ.},
     PAGES = {385--399},
 PUBLISHER = {Oxford Univ. Press, New York},
      YEAR = {1996},
      ISBN = {0-19-853862-6},
   MRCLASS = {03C60 (03C50)},
  MRNUMBER = {1428013},
MRREVIEWER = {G.\ Cherlin},
}

@article {dries1998corrections,
    AUTHOR = {Lou {\noopsort{dries}van den Dries}},
     TITLE = {Correction to: ``{$T$}-convexity and tame extensions. {II}''
              [{J}. {S}ymbolic {L}ogic {\bf 62} (1997), no. 1, 14--34;
              {MR}1450511 (98h:03048)]},
   JOURNAL = {J. Symbolic Logic},
  FJOURNAL = {The Journal of Symbolic Logic},
    VOLUME = {63},
      YEAR = {1998},
    NUMBER = {4},
     PAGES = {1597},
      ISSN = {0022-4812,1943-5886},
   MRCLASS = {03C60 (03C10 03C35 12L12)},
  MRNUMBER = {1665787},
       DOI = {10.2307/2586669},
       URL = {https://doi.org/10.2307/2586669},
}

@article {andujarguerrero2025marker,
    AUTHOR = {And\'ujar Guerrero, Pablo},
     TITLE = {The {M}arker-{S}teinhorn theorem},
   JOURNAL = {Notre Dame J. Form. Log.},
  FJOURNAL = {Notre Dame Journal of Formal Logic},
    VOLUME = {66},
      YEAR = {2025},
    NUMBER = {1},
     PAGES = {3--12},
      ISSN = {0029-4527,1939-0726},
   MRCLASS = {03C64},
  MRNUMBER = {4888212},
       DOI = {10.1215/00294527-2025-0001},
       URL = {https://doi.org/10.1215/00294527-2025-0001},
}

@book {rosenstein1982linear,
    AUTHOR = {Rosenstein, Joseph G.},
     TITLE = {Linear orderings},
    SERIES = {Pure and Applied Mathematics},
    VOLUME = {98},
 PUBLISHER = {Academic Press, Inc. [Harcourt Brace Jovanovich, Publishers],
              New York-London},
      YEAR = {1982},
     PAGES = {xvii+487},
      ISBN = {0-12-597680-1},
   MRCLASS = {06-02 (03C65)},
  MRNUMBER = {662564},
}

@misc{aschenbrenner2025analytic,
      title={Analytic Hardy fields}, 
      author={Matthias Aschenbrenner and Lou van den Dries},
      year={2025},
      eprint={2311.07352},
      archivePrefix={arXiv},
      primaryClass={math.LO},
      url={https://arxiv.org/abs/2311.07352}, 
}

@misc{aschenbrenner2025short,
      title={Short Hardy fields}, 
      author={Matthias Aschenbrenner and Lou van den Dries},
      year={2025},
      eprint={2508.06161},
      archivePrefix={arXiv},
      primaryClass={math.LO},
      url={https://arxiv.org/abs/2508.06161}, 
}

\end{document}